%% file: master.tex
\documentclass[11pt]{article}

\usepackage[margin=1in]{geometry}
\usepackage{amsmath,amsthm,amssymb,amsfonts,float}
\usepackage{listings}
\usepackage{pdfpages}
\usepackage{mathtools,slashed}
\usepackage{tikz}
\usetikzlibrary{patterns}

\makeatletter
\newcommand{\subjclass}[2][2010]{%
  \let\@oldtitle\@title%
  \gdef\@title{\@oldtitle\footnotetext{#1 \emph{Mathematics subject classification.} #2}}%
}
\makeatother

\usepackage[colorlinks=true, pdfstartview=FitV, linkcolor=blue,citecolor=blue, urlcolor=blue]{hyperref}

\usepackage[abbrev,lite,nobysame]{amsrefs}
\usepackage{color}

\usepackage{esint}

\usepackage{mathtools,enumitem,mathrsfs}

\usepackage[compact]{titlesec}

\usepackage{comment}

\input{macros.tex}

\renewcommand{\Re}{\textup{Re}}

\allowdisplaybreaks

\numberwithin{equation}{section}

\begin{document}

\title{Stability threshold of nearly-Couette shear flows \\ with Navier boundary conditions in 2D} 
\author{ Jacob Bedrossian\thanks{\footnotesize Department of Mathematics, University of California, Los Angeles, CA 90095, USA \href{mailto:jacob@math.ucla.edu}{\texttt{jacob@math.ucla.edu}}} \and Siming He\thanks{Department of Mathematics, University of South Carolina, Columbia, SC 29208, USA \href{mailto:siming@mailbox.sc.edu}{\texttt{siming@mailbox.sc.edu}}} \and Sameer Iyer\thanks{Department of Mathematics, University of California, Davis, Davis, CA 95616, USA \href{mailto:sameer@math.ucdavis.edu}{\texttt{sameer@math.ucdavis.edu}}} \and Fei Wang\thanks{School of Mathematical Sciences, CMA-Shanghai, Shanghai Jiao Tong University, 
		 Shanghai, China \href{mailto:fwang256@sjtu.edu.cn}{\texttt{fwang256@sjtu.edu.cn}}}}

\maketitle

\begin{abstract}
In this work, we prove a threshold theorem for the 2D Navier-Stokes equations posed on the periodic channel, $\mathbb{T} \times [-1,1]$, supplemented with Navier boundary conditions $\omega|_{y = \pm 1} = 0$. Initial datum is taken to be a perturbation of Couette in the following sense: the shear component of the perturbation is assumed small (in an appropriate Sobolev space) but importantly is independent of $\nu$. On the other hand, the nonzero modes are assumed size $O(\nu^{\frac12})$ in an anisotropic Sobolev space. For such datum, we prove nonlinear enhanced dissipation and inviscid damping for the resulting solution. The principal innovation is to capture quantitatively the \textit{inviscid damping}, for which we introduce a new Singular Integral Operator which is a physical space analogue of the usual Fourier multipliers which are used to prove damping. We then include this SIO in the context of a nonlinear hypocoercivity framework.  
\end{abstract}

\setcounter{tocdepth}{1}
{\small\tableofcontents}

\section{Introduction} \label{sec:intro}

\input{longtime.tex}

\section{Outline} \label{sec:outline}
\input{setuplittle.tex}

\section{Properties of the Inviscid Damping Energy Functional} \label{sec:tau}
\input{hypotau.tex}

\section{The Linearized Problem}  \label{sec:Lin}
\input{edlittle.tex}

\section{Nonlinear Estimates} \label{sec:NL}
\input{nonlinear.tex}

\vspace{4 mm}

\noindent \textbf{Acknowledgements:} JB was supported by NSF Award DMS-2108633. JB would also like to thank Ryan Arbon for helpful discussions. The work of  SH is supported in part by NSF grants DMS 2006660, DMS 2304392, DMS 2006372. The work of SI is partially supported by NSF DMS-2306528 and a UC Davis Society of Hellman Fellowship award. The work of FW is supported by the National Natural Science Foundation of China (No. 12101396, 12161141004, and 12331008).

\vspace{4 mm}

\noindent \textbf{Data Availability:} Data sharing not applicable to this article as no datasets were generated or analyzed for this work.


%
%
%
%
%
%


\addcontentsline{toc}{section}{References}
\bibliographystyle{abbrv}
\bibliography{bibliography}

\end{document}

%% file: macros.tex
\newcommand{\eps}{\epsilon}

\newcommand{\na}{\nabla}
\newcommand{\grad}{\nabla}

\newcommand{\norm}[1]{\left\|  #1 \right\|}
\newcommand{\abs}[1]{\left| #1 \right|}

\newcommand{\brak}[1]{\left\langle #1 \right\rangle} 

\newcommand{\rr}{\mathbb{R}}

\newcommand{\dee}{\mathrm{d}}

\newcommand{\dy}{\dee y}

\newcommand{\red}[1]{\textcolor{red}{#1}}

\usepackage{bbm}

\newcommand{\p}{\partial}

\newtheorem{theorem}{Theorem}[section]
\newtheorem{proposition}[theorem]{Proposition}

\newtheorem{lemma}[theorem]{Lemma}
\newtheorem*{lemma*}{Lemma}

\theoremstyle{definition}

\newtheorem{remark}[theorem]{Remark}

\def\abs#1{\left|#1\right|}
\newcommand{\n}{\ensuremath{\nonumber}}
\newcommand{\pa}{\ensuremath{\partial}}

\newcommand{\siming}[1]{{#1}} 

\newcommand{\nq}{{\neq}}

\newcommand{\myr}[1]{{ #1 }}

\newcommand{\al}{\alpha}

\newcommand{\de}{\Delta}
\newcommand{\lan}{\langle}
\newcommand{\ran}{\rangle}
\newcommand{\lf}{\left}
\newcommand{\rg}{\right}

\newcommand{\bb}{ \mathbb }

\newcommand{\dnt}{ {\delta_\ast\nu^{1/3}t}  }
\newcommand{\mh}[1]{}

%% file: longtime.tex
We study the 2D Navier-Stokes equations in the periodic channel $(x,y) \in \mathbb T \times [-1,1]\siming{= [-\pi,\pi]/{\sim}\times [-1,1]}$ with inhomogeneous Navier boundary conditions
\begin{subequations} \label{eq:FullNSIntro}
\begin{align} 
  & \partial_t v + (v \cdot \grad)v + \grad p = \nu \Delta v, \\
  & \grad \cdot v = 0,\quad  v_2(t,x,\pm 1) = 0, \quad
   \partial_y v_1(t,x,\pm 1) = 1,\\ 
  &v(t=0,x,y)=v_{in}(x,y). 
\end{align}
\end{subequations}
This problem corresponds to studying a fluid with Navier-type boundary conditions and applying a fixed force (rather than a fixed velocity) which slides the top and bottom boundaries in opposite directions.
See for example \cite{MasmoudiStRaymond03} for derivations of Navier-type boundary conditions from kinetic theory.
The boundary conditions on the vorticity $\Omega = \partial_y u_1 - \partial_x u_2$ reduce to simple Dirichlet conditions, giving the system
\begin{align*}
  & \partial_t \Omega + v \cdot \grad \Omega = \nu \Delta \Omega, \\
  & v = \grad^\perp {\Delta}^{-1} \Omega {=(\pa_y\de^{-1}\Omega,-\pa_x \de^{-1}\Omega)},\quad \Omega(t,x,y=\pm 1) = 1, \\
  & \Omega(t=0,x,y)=\Omega_{in}(x,y),
\end{align*}
where the $\Delta^{-1}$ is taken with homogeneous Dirichlet conditions on $y = \pm 1$ so that $v$ satisfies the no-penetration condition $v_2(t,x,\pm 1) = 0$.
It is straightforward to prove that all solutions of \eqref{eq:FullNSIntro} converge to the unique steady state selected by the boundary conditions, namely, the Couette flow:
\begin{align*}
v_0 = \begin{pmatrix} y \\ 0 \end{pmatrix}. 
\end{align*}
A class of questions which has received a lot of attention recently is that of a \emph{quantitative stability threshold}: given a norm on the initial data $\norm{\cdot}_X$, what is the largest $\gamma \geq 0$ such that
\begin{align*}
\norm{v(0) - v_0}_{X} \ll \nu^\gamma 
\end{align*}
implies that $v(t)$ behaves roughly like the linearized problem \emph{for all time}, where the exact quantification of this varies from work-to-work, but generally involves at least observing the enhanced dissipation characteristic of the linearized problem (discussed further below).
See \cite{BGM_Bull19} for a detailed discussion on thresholds in the context of the Couette flow. See e.g. \cite{Gallay18,CZEW20,CLWZ20,BVW16,MW20,MW19} and the references therein for 2D, and see 3D \cite{WZ21,BGM15I,BGM15II,BGM15III} and the references therein for 3D. 
When studying these kinds of high Reynolds number hydrodynamic stability problems, the mixing induced by the shearing greatly influences the stability of the equilibrium,
leading to two notable effects in 2D: \emph{inviscid damping} wherein the perturbation velocity decays in the linearized (or nonlinear) Euler equations and \emph{enhanced dissipation}, wherein the mixing accelerates the viscous dissipation, leading to, for example rapid decay of the $x$-dependence on time-scales like $\approx \nu^{-1/3}$ (rather than the decay of the heat equation $\approx \nu^{-1}$). Many works have studied these effects recently in the linearized Navier-Stokes and Euler equations; see for example \cite{Zillinger16,BCZV17,Jia2020,WZZ19,WZZ20,WZZ18,LinXu2017,ISJ22,J20,RWWZ23} and the references therein for inviscid damping results and e.g. \cite{Gallay18,J23,CZEW20,BH20,WZZ20,CWZ23} and the references therein for results on enhanced dissipation (some results study both, such as \cite{J23,CWZ23}, but such results are rare outside of the Couette flow). There are also works on mixing and enhanced dissipation by laminar flows in passive scalars, for example such as \cite{GCZ21,ABN22,BCZ15,BW13} and the references therein.  

The purpose of this paper is three-fold. The first two are:
(1) to introduce a new energy method for studying inviscid damping and enhanced dissipation; and (2) use this method to obtain a stability threshold not just for the Couette flow, but for all shear flows close to Couette, namely the `slowly-varying' solutions
\begin{align}
v(t) = \begin{pmatrix} e^{t \nu \Delta} v_0 \\ 0 \end{pmatrix}, \label{eq:vtslow}
\end{align}
with $\norm{v_0 - y}_{H^4} \leq \delta_0$ for a universal constant $\delta_0$ \emph{independent} of $\nu$.
In \cite{BVW16}, it was estimated that $\gamma \leq 1/2$ without boundaries (i.e. $y \in \mathbb R$) with initial data in Sobolev spaces $X = H^s$ with $s > 1$. 

For the case of the Couette flow (rather than \eqref{eq:vtslow}) more work exists.
In the case of $\mathbb T \times \mathbb R$, \cite{MW20} improves the required regularity for $\gamma \leq 1/2$ to $H^{log}_xL^2_y$.  And it turns out to be optimal for low regularity even 
up to $H^1_xL^2_y$~\cite{LiMasmoudiZhao22b}.
For $f_0 = y$ without boundaries, this was later improved to $\gamma \leq 1/3$ for sufficiently high Sobolev regularity \cite{MW19} and $\gamma\in[0,1/3]$ for suitable Gevrey classes,  which is expected to be sharp, due to the results of \cite{DM18}.
The regularity could be further relaxed to $H^s$ for $s\ge3$ in a upcoming paper \cite{WeiZhang23}. 
When one has at least Gevrey-2 regularity on the other hand, the result was proved with $\gamma = 0$ in \cite{BMV14} (i.e. the results are uniform-in-$\nu$). 
For $f_0 = y$ and Dirichlet boundary conditions in the channel, the stability threshold was proved with $\gamma = 1/2$ in \cite{CLWZ20}.
The method of that paper should also prove the corresponding result for the Navier boundary condition case, which is easier than the Dirichlet case.

In our upcoming work \cite{BHIW23I}, we extend the Gevrey-2 uniform-in-$\nu$ result of \cite{BMV14} to the channel in the case with Navier boundary conditions.
The proof works essentially in two steps: (1) to obtain precise Gevrey regularity estimates on time-scales $1 \ll t \ll \nu^{-{\zeta}}$ for some ${\zeta} > 1/3$ (which contain all of the results of e.g. \cite{HI20,BM13} and more) and then (2) apply the results of Theorem \ref{thm:main} below for times $t \gtrsim \nu^{-\zeta}$ to obtain a global-in-time result.
The third purpose of this paper is hence to complete the proof of the theorem stated in \cite{BHIW23I} by solving step (2). 
We have made this step a separate paper as the methods herein are largely different from \cite{BHIW23I}, and moreover, the result and the new energy method we introduce are of independent interest.  

Let us now make the results more precise. 
We will consider initial conditions of the form
\begin{align*}
\Omega_{in}(x,y) = 1 + W_{in}(y) + \omega_{in}(x,y), 
\end{align*}
where we will be assuming more regularity on $W_{in}$ but no $\nu$-dependent smallness, whereas $\omega_{in}$ will be assumed small relative to $\nu$.  \begin{subequations}\label{eq:omega}
We define the heat extension (with homogeneous Dirichlet conditions) of $W_{in}$ as
\begin{align}
W(t,y) = \lf(e^{\nu t \Delta_y} W_{in}\rg)(y),
\end{align}
and $U$ be the resulting shear flow given by the corresponding Biot-Savart law (here $G$ denotes the Greens function for $\partial_{yy}$ with homogeneous Dirichlet conditions on interval $[-1,1]$) 
\begin{align}
U(t,y) ={ y+}{ \partial_y \int_{-1}^1 G(y,y') W(t,y') \dee y'. }
\end{align}
We write the solution of the Navier-Stokes equations as
\begin{align*}
\Omega ={1+} W(t,y) + \omega(t,x,y), 
\end{align*}
reducing the problem to 
\begin{align} 
  & \partial_t \omega + U \partial_x \omega {- U'' }\partial_x \phi  + u \cdot \grad \omega = \nu \Delta \omega,  \\
  & u = \grad^\perp \Delta^{-1} \omega, \ \
    \omega(t,y=\pm 1) = 0, \ \ 
   \omega(0) = \omega_{in}, 
\end{align}
\end{subequations}
which now contains time-dependent linear terms which are not perturbatively small in $\nu$ and do not satisfy any straightforward energy estimates because of the non-local term $U''\partial_x\phi$ (note that $U''=\pa_{yy}U$ is not sign-definite).
Nevertheless, we obtain a new energy method capable of dealing with the time-dependent linearized Navier-Stokes equations for Navier boundary conditions and extracting both inviscid damping and enhanced dissipation at the same time (the same method works, and is simpler, in $\mathbb T \times \mathbb R$).
The approach is nearly a physical-side analogue of the Fourier multiplier energy method employed in \cite{BMV14} to solve the corresponding problem without a boundary. 
Such a Fourier multiplier approach cannot be employed here due to the boundary, however, one can find a singular integral operator that is essentially a physical-side analogue of the multiplier in \cite{BMV14} used therein to  obtain the fundamental $L^2_t H^{1}_x L^2_y$ inviscid damping estimate on the velocity field and which allows to integrate the $U''\partial_x\phi$ linear term.
Using this singular integral operator, combined with a more standard hypocoercivity method to obtain the enhanced dissipation and suitable nonlinear estimates, we obtain the following theorem purely via an energy method. 
\begin{theorem} \label{thm:main}
Consider the equation \eqref{eq:omega} subject to compatibility condition $W_{in}\big|_{y=\pm 1}=\omega_{in}\big|_{y=\pm 1}\equiv0.$ Then for all $m \in (2/3,1)$, $\exists \delta_0,\delta_1 > 0$ such that if
\begin{align}\label{asmp}
&{\norm{W_{in}}_{ H_y^4} }  \leq \delta_0, \quad \
 \sum_{0 \leq j \leq 1} \norm{(\nu^{1/3} \partial_y)^j \brak{\partial_x}^{m - \frac{j}{3}} \omega_{in}}_{L^2}  =: \eps \leq \delta_1\sqrt{\nu}, 
\end{align}
then 
all $\delta_\ast > 0$ sufficiently small (depending only on universal constants) and all $\nu \in (0,1)$, there hold
\begin{align*}
\sum_{0 \leq j \leq 1} &\norm{ \brak{\partial_x}^{m - \frac{j}{3}} (\nu^{1/3}\partial_y)^j\omega_{\neq}(t)}_{L_{x,y}^2} \lesssim\mh{_\delta} \epsilon e^{-\delta_\ast \nu^{1/3} t},\quad \forall t\in[0,\infty); \\
\sum_{0 \leq j \leq 1}  &\norm{ (\nu^{1/3} \partial_y)^j \omega_0(t)}_{L_y^2} \lesssim\mh{_\delta} \eps e^{-\delta_\ast \nu t},\quad \forall t\in[0,\infty); \\
\sum_{0\leq j \leq 1}  &\norm{e^{\delta_\ast \nu^{1/3} t}\abs{\partial_x}^{m+1\mh{-\delta} - \frac{j}{3}} (\nu^{1/3} \partial_y)^j u_{\neq}}_{L^2_t L^2_{x,y}} \lesssim\mh{_\delta} \eps. 
\end{align*}
\end{theorem}

\ifx
\begin{remark}The condition \eqref{asmp} guarantees that 
\begin{align}
\|W\|_{L^\infty_tH_y^4}\leq C\delta_0.
\end{align}
This is because the $W$ is the heat extension of $W_{in}$ and the initial data is consistent, i.e., $W_{in}(y=\pm 1)=\omega_{in}(x,y=\pm 1)\equiv 0.$ 
If the initial data is not consistent, there might be a jump in the Sobolev norms. An alternative assumption can be that $W_{in}(y=\pm 1)=0$ and $\int_{\mathbb{T}}\omega_{in}(x,y)\dee x\equiv0.$
\end{remark}
\fi
\begin{remark}
Note that while we require a little more than $2/3$ of a derivative in $x$ to be $O(\sqrt{\nu})$ in $L^2$, we only require $\partial_y\omega$ to be $O(\nu^{1/6})$ in $H^{1/3+}_x L^2_y$.
\end{remark}

\begin{remark}
In the case of Dirichlet boundary conditions, the boundary layers will necessitate a more complicated approach that cannot be done only with energy estimates. Nevertheless, the proofs of \cite{CLWZ20} show that a strong understanding of the Navier case is extremely useful for making progress on the Dirichlet case. The case of Dirichlet boundary conditions will be considered in future work. 
\end{remark}


%% file: setuplittle.tex


\subsection{Linearized Problem}
One of the primary challenges to Theorem \ref{thm:main} is obtaining sufficiently good estimates on the (time-dependent) linearized problem that results from dropping the nonlinearity in \eqref{eq:omega}.
We believe these could be  obtained using a time-splitting argument similar to that used in \cite{ChenWeiZhang20} along with a suitable variation of the resolvent estimates found in \cite{CLWZ20}. 
While effective and robust, this method is also quite technical.
One of our primary contributions here is to provide a more straightforward method based solely on energy estimates. 

Due to translation invariance, the linearized problem is best studied mode-by-mode in $x$. 
Denoting $\displaystyle\omega_k(t,y) := {\frac{1}{2\pi}} \int_{-\pi}^{\pi} \omega(t,x,y) e^{-ikx} dx$, we have the mode-by-mode linearization given by 
\begin{subequations} \label{NSchannel}
\begin{align}
&  \pa_t \omega_k + U ik \omega_k {- U''ik}\phi_k - \nu \Delta_k \omega_k= 0, \\
  &   \Delta_k \phi_k =  \omega_k, \\
& \phi_k(\pm 1) = 0, \quad\ \omega_k(\pm 1 )= 0, 
\end{align}
\end{subequations}
where $\Delta_k := -\abs{k}^2 + \partial_{yy}$. 
We introduce a hypocoercive energy for \eqref{NSchannel}, which simultaneously captures the inviscid damping and the enhanced dissipation:
for universal constants $\{c_\ast\}_{\ast\in\{\al,\,\beta,\, \tau\}} \subset (0,1)$ to be specified below, 
\begin{align}\n
E_k[\omega_k] :=&  \| \omega_{k} \|_{L^2}^2   + c_{\alpha} \nu^{\frac23} |k|^{-\frac23}  \| \p_y \omega_{k} \|_{L^2}^2 - c_{\beta} \nu^{\frac{1}{3}} |k|^{-\frac43}  \Re \langle ik \omega_{k}, \p_y \omega_{k} \rangle \\ 
&   + c_{\tau} \Re \langle \omega_k, \mathfrak{J}_k[\omega_k] \rangle +c_{\tau} c_\alpha  \nu^{\frac23} \abs{k}^{-\frac23} \Re \langle \p_y \omega_k, \mathfrak{J}_k[\p_y \omega_k] \rangle \n\\
=&\Re\lan \omega_k ,(1+c_\tau\mathfrak{J}_k)[\omega_k]\ran +c_\al\nu^{\frac{2}{3}}|k|^{-\frac{2}{3}}\Re\lan \pa_y\omega_k ,(1+c_\tau\mathfrak{J}_k)[\pa_y \omega_k]\ran - c_{\beta} \nu^{\frac{1}{3}} |k|^{-\frac43}  \Re \langle ik \omega_{k}, \p_y \omega_{k} \rangle . \label{def:Ek}
\end{align}
Here the singular integral operator (SIO) $\mathfrak{J}_k$ is given by
\begin{align}
\mathfrak{J}_k[f](y) := \abs{k}\mh{^{1-\delta}} \text{p.v.} \frac{k}{\abs{k}} \int_{-1}^1 \frac{1}{2i(y-y')} G_k(y, y') f(y') \dee y', 
\end{align}
where we denote $G_k$ the Green's function for $\Delta_k:=-|k|^2+\pa_{yy}$ with homogeneous Dirichlet boundary conditions, which has the explicit form: 
\begin{align} 
G_k(y,y')=-\frac{1}{k\sinh(2k)}\left\{\begin{array}{cc}\sinh(k(1-y'))\sinh(k(1+y)),&\quad y\leq y';\\
\sinh(k(1-y))\sinh(k(1+y')),&\quad y\geq y'.\end{array}\right.\label{G_k}
\end{align}
The first three terms in \eqref{def:Ek} are the by-now standard hypocoercive energy for extracting optimal enhanced dissipation from problems such as passive scalar shear flows (i.e. \eqref{NSchannel} with no $ik U'' \phi_k$ term); see for example \cite{GCZ21,BCZ15}.   
The last two terms in the energy simultaneously deal with the non-local $ikU''\phi_k$ in the energy estimate (which would normally frustrate any straightforward hypocoercivity estimate) and also extracts a $L^2_t \dot{H}^{1-\delta}_x L^2_y$ inviscid damping estimate which is nearly the optimal estimate expected from the 2D Couette flow \cite{CLWZ20}, namely that
\begin{align*}
\abs{k}^2 \norm{\grad_k \phi_k}_{L^2_t L^2_y}^2 \lesssim \norm{\omega_k(0)}_{L^2}^2. 
\end{align*}
In our case, the terms involving $\mathfrak{J}_k$ will provide the following invisicd damping estimates 
\begin{align*}
\sum_{0 \leq j \leq 1} \abs{k}^{2\mh{-\delta}} \norm{(\nu^{1/3} \abs{k}^{-1/3} \partial_y)^j \grad_k \phi_k}_{L^2_t L^2_y}^2  \lesssim E_k[\omega_{k}(0)]. 
\end{align*}
We show below that $\mathfrak{J}_k$ is a bounded operator, mapping $L^2 $ to $L^2$, with a bound depends uniformly in $k$, and so for appropriate choices of the parameters (see Section \ref{sec:FinLin} for the details on how to set the parameters $\{c_\ast\}_{\ast\in\{\al,\beta,\tau\}}$), we have
\begin{align*}
E[\omega_k] \approx \norm{\omega_k}_{L^2}^2 + \frac{\nu^{2/3}}{\abs{k}^{2/3}} \norm{ \partial_y \omega_k}_{L^2}^2,
\end{align*}
with an implicit constant independent of $\nu$ and $k$. 

The first main result of our work is the following linearized energy estimate, proved in Section \ref{sec:Lin}.
\begin{proposition} \label{prop:lin}
For suitable choices of the parameters $\{c_\ast\}_{\ast\in\{\al,\beta,\tau\}}$ \eqref{def:Ek}, and for any $H^1$ solutions to $\eqref{NSchannel}_{k \neq 0}$, the following holds for all sufficiently small $\delta_\ast>0$ that are independent of $\nu$ and $k$,
\begin{align*}
e^{2\delta_\ast \nu^{1/3} \abs{k}^{2/3} t} E_k[\omega_k(t)] + \frac{1}{4}c_\tau \int_0^t e^{2\delta_\ast \nu^{1/3} \abs{k}^{2/3} s} \abs{k}^{2\mh{-\delta}} \norm{\grad_k \phi_k(s)}_{L^2}^2 ds  \leq E_k[\omega_k(0)],  \quad \forall t \geq 0,
\end{align*}
and, more specifically, the following energy estimate holds: 
\begin{align*}
\frac{d}{dt} E_k[\omega_k] \leq   -8\delta_\ast\mathrm{D}_k[\omega_k] - 8\delta_\ast \nu^{1/3} \abs{k}^{2/3} E_k[\omega_k],
\end{align*}
where $\mathrm{D}_k$ is given by
\begin{subequations} \label{def:Dks}
\begin{align}
\mathrm{D}_k & := \mathrm{D}_{k,\gamma} + c_\alpha \mathrm{D}_{k,\alpha} + c_\beta \mathrm{D}_{k,\beta} + c_\tau \mathrm{D}_{k,\tau} + c_\tau c_\alpha \mathrm{D}_{k,\tau\alpha}, \\ 
\mathrm{D}_{k,\gamma} & := \nu \norm{\grad_k \omega_k}_{L^2}^2:=\nu\|ik\omega_k\|_{L^2}^2+\nu\|\pa_{y}\omega_k\|_{L^2}^2, \\
\mathrm{D}_{k,\alpha} & := \nu  \abs{k}^{-2/3} \norm{\grad_k (\nu^{1/3}\partial_y) \omega_k}_{L^2}^2, \\
\mathrm{D}_{k,\beta} & := \nu^{1/3} \abs{k}^{2/3} \norm{\omega_k}_{L^2}^2, \\
\mathrm{D}_{k,\tau} & := \abs{k}^{2\mh{-\delta}} \norm{\grad_k \phi_k}_{L^2}^2, \\
\mathrm{D}_{k,\tau \alpha} & := \nu^{2/3} \abs{k}^{4/3\mh{-\delta}} \norm{\grad_k \partial_y \phi_k}_{L^2}^2. 
\end{align}
\end{subequations}
Here $\na_k :=(ik,\pa_y).$ Furthermore, we define the notation $\mathrm{D}_{\ast}:=\sum_{k\neq 0}\mathrm{D}_{k,\ast},\, \ast\in\{\gamma,\al,\beta,\tau,\tau\al\}.$ 
\end{proposition}

The definition of $\mathfrak{J}_k$ comes from a physical-side representation of a Fourier multiplier employed in the works \cite{BGM15I,BMV14, BVW16,Zillinger2014}.
In the case of $\mathbb{T}\times\rr$, a similar physical-side SIO as $\mathfrak{J}_k$ was used in the recent work \cite{LiZhao23}.
In \cite{BMV16}, the Fourier multiplier was employed for the same purpose of dealing with the non-local term in the linearized equation \eqref{NSchannel}. 
In the work \cite{Zillinger2014,BVW16}, the problem was re-written in the coordinates $f(t,z,y) = \omega_k(t,z+ty,y)$ and a norm based on the following energy estimate for the linearized Euler equations was made
\begin{align*}
\frac{d}{dt}\norm{\mathcal{M}(t,k,\partial_y) f_k}_{L^2}^2 + (1-C\delta_0)\norm{\sqrt{\frac{-\partial_t \mathcal{M}}{\mathcal{M}}} \mathcal{M} f_k}_{L^2}^2 \leq 0, 
\end{align*}
where the Fourier multiplier $\mathcal{M}$ is chosen such that 
\begin{align*}
& \mathcal{M}(0,k,\eta) = 1, \qquad \ 
 \frac{d}{dt}\mathcal{M}(t,k,\eta) = -\frac{\abs{k}^2}{k^2 + \abs{\eta - tk}^2} \mathcal{M}(t,k,\eta). 
\end{align*}
Undoing the coordinates shows that the term involving $\partial_t \mathcal{M}/\mathcal{M}$ is in fact $\abs{k}^2 \norm{\grad_k \phi_k}^2_{L^2}$. 
The operator $\mathcal{M}(t,k,\partial_y)$ is order zero and $\norm{\mathcal{M} f}_{L^2} \approx \norm{f}_{L^2}$, which has led to the use of the word ``ghost multiplier'' (in analogy with the ghost energy method of Alinhac \cite{Alinhac01}).
One explicitly computes that (using that $\arctan$ is odd),  
\begin{align*}
\mathcal{M}(t,k,\eta) = \exp \left( \int_0^{t} \frac{1}{1 + \abs{\tau - \frac{\eta}{k}}^2} \dee \tau \right) = \exp \left( \arctan(t + \frac{\eta}{k}) - \arctan(\frac{\eta}{k}) \right).
\end{align*} 
Undoing the shift, 
\begin{align*}
\mathcal{M}(t,k,\eta-tk) = \exp\left( \arctan\frac{\eta}{k} - \arctan(\frac{\eta}{k} - t) \right) \to \exp\left( \arctan\frac{\eta}{k} + \pi/2\right).  
\end{align*}
For $\eta/k > 0$ as $\eta \to \infty$ this expands as
\begin{align*}
\exp\left( \arctan\frac{\eta}{k} + \pi/2\right) \approx e^{\pi}\left(1 + \arctan\frac{\eta}{k} + \pi/2 + ... \right). 
\end{align*}
We observe that the inverse Fourier transform of $\arctan$ is given by 
\begin{align*}
\int_{\mathbb R} e^{i\eta y} \arctan \frac{\eta}{k} d\eta =  -\frac{1}{iyk}\int_{\mathbb R} e^{i\eta y} \frac{\abs{k}^2}{\abs{k}^2 + \abs{\eta}^2} d\eta = -\frac{k}{iy} G_k(y), 
\end{align*}
where $G_k$ would be the fundamental solution of $-\Delta_k$ on $\mathbb R$. 
The resulting singular integral operator should be the most important aspect of $\mathcal{M}$ on the physical-side.
This motivates the definition of $\mathfrak{J}_k$, though it is not a priori clear why using $\mathfrak{J}_k$ should be sufficient to deduce inviscid damping, even without further adaptation to the time-dependent shear flow. 
The properties of $\mathfrak{J}_k$ are outlined in Section \ref{sec:tau} while the proof that its use provides the desired inviscid damping is in Section \ref{sec:LinTau}. 

\subsection{Nonlinear Problem}
For the nonlinear problem we use the following natural nonlinear extension of the linear energy functional (with the time-decay built in):
\begin{subequations}
\begin{align}
  \mathcal{E}_0 & := e^{2\delta_\ast \nu t} \| \omega_{0} \|_{L^2}^2   +  c_{\alpha}e^{2\delta_\ast \nu t} \| \nu^{1/3} \p_y  \omega_{0} \|_{L^2}^2,\label{mcl_E_0}  \\
  \mathcal{E}_{\neq} & := \sum_{k \neq 0}  e^{2\delta_\ast \nu^{1/3} t} \abs{k}^{2m} E_k[\omega_k], \label{mcl_E_nq} \\ 
  \mathcal{E} & := \mathcal{E}_0 + \mathcal{E}_{\neq}.\label{mcl_E} 
\end{align}
\end{subequations}
This energy comes with the associated dissipation functional\begin{subequations}
\begin{align}
  \mathcal{D}_0 & := e^{2\delta_\ast \nu t} \nu \| \partial_y \omega_{0} \|_{L^2}^2   + c_{\alpha}   e^{2\delta_\ast \nu t}\nu   \|  \nu^{1/3}{\partial_y^2}  \omega_{0} \|_{L^2}^2,\label{mcl_D_0} \\
  \mathcal{D}_{\neq} & := \sum_{k \neq 0} e^{2\delta_\ast \nu^{1/3} t} \abs{k}^{2m} \mathrm{D}_k[ \omega_k].\label{mcl_D_nq}
\end{align}
\end{subequations}
We also use the corresponding notation for the dissipation operators such as $\mathrm{D}_{k,\gamma}$, $\mathrm{D}_{k,\alpha}$, etc. as in \eqref{def:Dks}  
\begin{align}\label{def:Dkscl}  
\mathcal{D}_{(\ast)} := \sum_{k \neq 0} e^{2\delta_\ast \nu^{1/3} t} \abs{k}^{2m} \mathrm{D}_{k,(\ast)}[ \omega_k]. 
\end{align}
We note that the definitions $\mathcal{D}_\ast$ and $\mathrm{D}_\ast$ differ by $|k|^{2m}$ factor and the time weight $e^{2\delta_\ast\nu t},e^{2\dnt}$. 
For future convenience we also define
\begin{align}\label{D_E}
\mathcal{D}_{E} := \nu \mathcal{E}_0 + \nu^{1/3} \sum_{k \neq 0}  e^{2\delta_\ast \nu^{1/3} t} \abs{k}^{2m + 2/3} E_k[\omega_k]. 
\end{align}
Finally, we define
\begin{align}
  \mathcal{D}&:=\mathcal{D}_0+\mathcal{D}_\nq+\mathcal{D}_E. \label{mcl_D}
\end{align}
Theorem \ref{thm:main} is an immediate consequence of the following energy estimate, proved in Section \ref{sec:NL}.
\begin{proposition} \label{prop:MainBoot} 
There exists a constant $C_0$, depending only on the parameters $\delta_\ast$, $\delta_0$, and $m\in(2/3,1)$ such that 
\begin{align}
  \label{ineq:MainEst}
\frac{d}{dt} \mathcal{E} + 4\delta_\ast\mathcal{D}  \leq \left(C_0 \nu^{-1}\mathcal{E}\right)^{1/2} \mathcal{D}.
\end{align}
Here $\mathcal{E}$ and $\mathcal{D}$ are defined in \eqref{mcl_E} and \eqref{mcl_D}. 
In particular, the bound \eqref{ineq:MainEst} implies that if $\mathcal{E}|_{t=0} \leq \frac{\delta_\ast^2}{C_0} \nu$, then the following global estimate 
\begin{align*}
\sup_{t \geq 0} \mathcal{E}(t) + 2\delta_\ast\int_0^\infty \mathcal{D}(s) \dee s \leq \mathcal{E}(0)
\end{align*}
holds. 
\end{proposition}

\subsection{Notations}
Throughout the paper, the constant $C$ will change from line to line. For $A,B\geq 0$, we use the notation $A\lesssim B$ to highlight that there exists a constant $C>0$ such that $A\leq CB$. 

To avoid introducing too many symbols, we further define `local variables', denoted as $T_{\cdots}$, for example, $T_{0\nq}$ and $T_{\nq0}$. These notations represent terms that appear during the estimation in each subsection. Once a subsection is concluded, these notations will be redefined in the next subsection.

%% file: hypotau.tex
Here we develop some properties on the operator $\mathfrak{J}$. 
For future notational convenience, we denote
\begin{align}
\mathfrak{J}_{k,\eps}[f] := \abs{k} \frac{k}{\abs{k}} \int_{-1}^{y-\eps} \frac{G_k(y, y') }{2i(y-y')} f(y') dy' + \abs{k} \frac{k}{\abs{k}} \int_{y+\eps}^1 \frac{ G_k(y, y')}{2i(y - y')} f(y') dy', 
\end{align}
so that by definition,  
\begin{align}
\mathfrak{J}_k[f] = \lim_{\eps \rightarrow 0} \mathfrak{J}_{k,\eps}[f]. 
\end{align}
We first prove that $\mathfrak{J}_k$ extends to a bounded linear operator $\mathfrak{J}_k:L^2 \to L^2$. 

\begin{lemma} \label{lem:BoundI}
The singular integral operator $\mathfrak{J}_k$ extends to a bounded linear operator on $L^2 \to L^2$ and moreover
\begin{align*}
\norm{\mathfrak{J}_k}_{L^2 \to L^2} \lesssim 1. 
\end{align*}
\end{lemma}
\begin{remark}
Note that we therefore have $\forall f \in L^2$,  $\lim_{\eps \to 0} \norm{\mathfrak{J}_k[f] - \mathfrak{J}_{k,\eps}[f]} = 0$. 
\end{remark}

\begin{proof}We decompose 
\begin{align*}
\frac{1}{|k|}\mathfrak{J}_{k}f  = &  \frac{k}{\abs{k}} \textup{p.v.}  \int_{-1}^1 G_k(y,y') \frac{f(y')}{2i(y-y')} dy' \\   
= & \frac{k}{\abs{k}} \textup{p.v.}  \int_{[-1,1] - (y - \frac{1}{k}, y + \frac{1}{k})} G_k(y,y') \frac{f(y')}{2i(y-y')} dy' +  \frac{k}{\abs{k}} \textup{p.v.}  \int_{y - \frac{1}{k}}^{y + \frac{1}{k}} G_k(y,y') \frac{f(y')}{2i(y-y')} dy' \\
 = &\frac{k}{\abs{k}} \textup{p.v.}  \int_{[-1,1] - (y - \frac{1}{k}, y + \frac{1}{k})} G_k(y,y') \frac{f(y')}{2i(y-y')} dy'  +  \frac{k}{\abs{k}} G_k(y,y) \textup{p.v.}  \int_{y - \frac{1}{k}}^{y + \frac{1}{k}} \frac{f(y')}{2i(y-y')} dy' \\
 & +   \frac{k}{\abs{k}}  \textup{p.v.} \int_{y - \frac{1}{k}}^{y + \frac{1}{k}} [G_k(y,y') - G_k(y,y)]  \frac{f(y')}{2i(y-y')} dy' \\
 = & \mathcal{T}_1 + \mathcal{T}_2 + \mathcal{T}_3. 
\end{align*}
The strategy is as follows: for $\mathcal{T}_1$, we use a Schur-type estimate and the rapid decay of $G_k(y,y')$ to gain integrability. For $\mathcal{T}_2$, we use properties on the (truncated) Hilbert transform. For $\mathcal{T}_3$, we use $L^\infty_y L^1_{y'} + L^\infty_{y'} L^1_y$ bounds to gain the $1/k$. 

\vspace{2 mm}

\noindent \textit{Bound of $\mathcal{T}_1:$} The bound of $\mathcal{T}_1$ is perhaps the most subtle, because it relies on the properties of $G_k$. Nevertheless, we may split $\mathcal{T}_1 = \mathcal{T}_{1,\le} + \mathcal{T}_{1,\ge}$, where 
\begin{align}
\mathcal{T}_{1,\le} := & \frac{k}{\abs{k}} \textup{p.v.}  \int_{-1}^{y - \frac{1}{k}} G_k(y,y') \frac{f(y')}{2i(y-y')} dy'  \\
\mathcal{T}_{1,\ge} := & \frac{k}{\abs{k}} \textup{p.v.}  \int_{y + \frac{1}{k}}^1 G_k(y,y') \frac{f(y')}{2i(y-y')} dy'.
\end{align}
These two operators are treated analogously, so we focus on $\mathcal{T}_{1,\ge}$. First of all, we notice that we can remove the principle value due to excising the point $y' = y$, and we therefore can write 
\begin{align}
\mathcal{T}_{1,\ge} f = \frac{k}{|k|} \int \frac{G_{k}(y,y')}{2i(y-y')} \mathbbm{1}_{y + \frac{1}{k}}(y') f(y') dy' = \int_{-1}^1 \bold{K}_k(y, y') f(y') dy' 
\end{align}
By using Schur's test, we have 
\begin{align} \label{Sch:1}
\| \mathcal{T}_{1,\ge} f \|_{L^2} \lesssim (\| \bold{K}_k \|_{L^\infty_{y} L^1_{y'}} + \| \bold{K}_k \|_{L^\infty_{y'} L^1_{y}}) \| f \|_{L^2}
\end{align}
Therefore, we need to estimate
\begin{align}
\| \bold{K}_k(y,\cdot)  \|_{ L^1_{y'}} = \int_{y + \frac{1}{k}}^1 \frac{|G_k(y,y')|}{|y - y'|} dy' \lesssim & \int_{y + \frac{1}{k}}^1 \frac{\sinh(k(1-y')) \sinh(k(1 + y))}{k|y-y'| \sinh(2k)} dy' \\
\lesssim & \int_{y + \frac{1}{k}}^1 \frac{[e^{k(1 - y')}  - e^{-k(1-y')} ][ e^{k(1 + y)} - e^{-k(1+y)}  ] }{k|y-y'| e^{2k}} dy' \\
= & \int_{y + \frac{1}{k}}^1 \frac{e^{k(1 - y')}  e^{k(1 + y)}  }{k|y-y'| e^{2k}} dy'  + \text{Err}_1,
\end{align}
where the error integral above is defined as follows: 
\begin{align}
\text{Err}_1 := \int_{y + \frac1k}^1 \frac{e^{-k(1-y')}[e^{k(1 + y)} - e^{-k(1 + y)}]}{k|y' - y| e^{2k}} + \int_{y + \frac1k}^1 \frac{ [ e^{k(1-y')} - e^{-k(1-y')} ] e^{-k(1 + y)} }{k|y - y'|e^{2k}}
\end{align}
The main contribution simplifies as follows 
\begin{align}
\int_{y + \frac{1}{k}}^1 \frac{e^{k(1 - y')}  e^{k(1 + y)}  }{k|y-y'| e^{2k}} dy' = \int_{y + \frac{1}{k}}^1 \frac{e^{-k( y' - y)}  }{k(y' - y)} dy' = \frac{1}{k} \int_{u = 1}^{k(1-y)} \frac{e^{-u}}{u} du \lesssim \frac{1}{k}.
\end{align}
The bounds on the term $\text{Err}_1$ is simpler, and clearly also satisfies the bound above simply by using the following bound on the numerators:
\begin{align*}
e^{-k(1-y')}[e^{k(1 + y)} - e^{-k(1 + y)}] \lesssim &e^{k(1 - y')} e^{k(1 + y)}, \\
[ e^{k(1-y')} - e^{-k(1-y')} ] e^{-k(1 + y)} \lesssim & e^{k(1 - y')} e^{k(1 + y)}.
\end{align*}
Taking the supremum in $y$, we obtain the bound $\| \bold{K}_k  \|_{L^\infty_y L^1_{y'}} \lesssim 1/k$.  

We now need to estimate the quantity $\| \bold{K}_k  \|_{L^\infty_{y'}L^1_y} \lesssim 1/k$. We therefore fix a $y'$ in the range $1/k \le y' \le 1$. We then have 
\begin{align*}
\| \bold{K}_k(\cdot, y')  \|_{L^1_y} = & \int \frac{G_k(y,y')}{(y' - y)} \mathbbm{1}(y \le y'  - \frac{1}{k}) dy = \int_{y = 0}^{y' - \frac1k} \frac{G_k(y, y')}{|y - y'|} dy \\
\lesssim & \int_{y = 0}^{y' - \frac{1}{k}} \frac{  e^{k(1-y')} e^{k(1 + y)}  }{k(y'  - y) e^{2k}} dy  = \int_{y = 0}^{y' - \frac{1}{k}} \frac{e^{-k(y' - y)}}{k(y' - y)} dy \\
\lesssim & \frac{1}{k} \int_1^{ky'} \frac{e^{-u}}{u} du \lesssim \frac{1}{k}.
\end{align*}

Therefore, inserting these kernel bounds into \eqref{Sch:1}, we have 
\begin{align} \n 
\| \mathcal{T}_{1,\ge} f \|_{L^2} \lesssim (\| \bold{K}_k \|_{L^\infty_{y} L^1_{y'}} + \| \bold{K}_k \|_{L^\infty_{y'} L^1_{y}}) \| f \|_{L^2} \lesssim \frac{1}{k} \| f \|_{L^2}.
\end{align}
By applying the analogous argument also to $\mathcal{T}_{1,\le}$, we have 
\begin{align} \label{T1bd}
\| \mathcal{T}_{1} f \|_{L^2} \lesssim \frac{1}{k} \| f \|_{L^2}.
\end{align}
\vspace{2 mm}

\noindent \textit{Bound of $\mathcal{T}_2:$} For the operator $\mathcal{T}_2$, we use the boundedness properties of the truncated Hilbert transform. In particular, we proceed as follows:
\begin{align} \n
\| \mathcal{T}_2f\|_{L^2} = & \| \frac{k}{\abs{k}} G_k(y,y) \textup{p.v.}  \int_{y - \frac{1}{k}}^{y + \frac{1}{k}} \frac{f(y')}{2i(y-y')} dy' \|_{L^2} \lesssim \| G_k(y, y) \|_{L^\infty_{y}} \|  \textup{p.v.}  \int_{y - \frac{1}{k}}^{y + \frac{1}{k}} \frac{f(y')}{2i(y-y')} dy' \|_{L^2} \\ \label{htrn:1}
\lesssim & \frac{1}{k} \| \textup{p.v.}  \int_{y - \frac{1}{k}}^{y + \frac{1}{k}} \frac{f(y')}{2i(y-y')} dy' \|_{L^2}.
\end{align}
It remains thus to estimate the $L^2 \rightarrow L^2$ boundedness of the operator above. To do so, we realize this as the difference between the classical Hilbert transform, and the truncated Hilbert transform. Indeed, define 
\begin{align}
\mathcal{H}f := &\textup{p.v.}  \int_{-1}^{1} \frac{f(y')}{2i(y-y')} dy' \\
\mathcal{H}_{\delta} f := &  \textup{p.v.}  \int_{-1}^{y-\delta} \frac{f(y')}{2i(y-y')} dy' + \int_{y+\delta}^{1} \frac{f(y')}{2i(y-y')} dy', \\
\overline{\mathcal{T}}_{2} := & \textup{p.v.}  \int_{y - \frac{1}{k}}^{y + \frac{1}{k}} \frac{f(y')}{2i(y-y')} dy' 
\end{align}
{with the convention that 
	\begin{align*}
		\textup{p.v.}\int_{-1}^{y-\delta} \frac{f(y')}{2i(y-y')} dy'=0, &\ \ \ \mbox{if}\  y-\delta<-1;\\
		\textup{p.v.} \int_{y+\delta}^{1} \frac{f(y')}{2i(y-y')} dy' = 0, &\ \ \ \mbox{if}\  y+\delta >1.
	\end{align*}}
Then, we have 
\begin{align*}
\overline{\mathcal{T}}_{2} = \mathcal{H}f - \mathcal{H}_{1/k}f,  
\end{align*}
from which the $L^2 \rightarrow L^2$ boundedness of $\overline{\mathcal{T}}_2$ follows from the classical corresponding boundedness estimates (uniform in $\delta$): 
\begin{align*}
\| \mathcal{H}f \|_{L^2} \lesssim \| f \|_{L^2}, \qquad \| \mathcal{H}_{\delta}f \|_{L^2} \lesssim \| f \|_{L^2}.
\end{align*}
Therefore, continuing from \eqref{htrn:1}, we have 
\begin{align} \label{T2bd}
\| \mathcal{T}_2f\|_{L^2} \lesssim \frac{1}{k} \| \overline{\mathcal{T}}_2f\|_{L^2} \lesssim & \frac{1}{k} \| f \|_{L^2}.
\end{align}
\vspace{2 mm}

\noindent \textit{Bound of $\mathcal{T}_3:$} For the operator $\mathcal{T}_3$, we use Schur's Test again. In this case, we have 
\begin{align*}
\int_{y - \frac{1}{k}}^{y + \frac{1}{k}} \frac{|G_k(y, y') - G_k(y,y)|}{|y - y'|} dy' \lesssim \sup_{y, y'} \| \frac{G_k(y,y') - G_k(y,y)}{|y - y'|} \Big| \int_{y - \frac{1}{k}}^{y + \frac1k} dy' \lesssim \frac{1}{k}, \\
\int_{y' - \frac{1}{k}}^{y' + \frac{1}{k}} \frac{|G_k(y, y') - G_k(y,y)|}{|y - y'|} dy \lesssim \sup_{y, y'} \| \frac{G_k(y,y') - G_k(y,y)}{|y - y'|} \Big| \int_{y' - \frac{1}{k}}^{y' + \frac1k} dy \lesssim \frac{1}{k}. 
\end{align*}  
Therefore, again by applying Schur's test, we have 
\begin{align} \label{T3bd}
\| \mathcal{T}_3f\|_{L^2}  \lesssim & \frac{1}{k} \| f \|_{L^2}.
\end{align}
Combining the bounds \eqref{T1bd}, \eqref{T2bd}, \eqref{T3bd}, we obtain the desired result. 
\end{proof}


The next lemma captures the commutator between $\partial_y$ and $\mathfrak{J}_k$. 
\begin{lemma} \label{lem:CommIpy}
If $f \in H^1$ then $\mathfrak{J}_k[f] \in H^1$ we define 
\begin{align*}
[\partial_y, \mathfrak{J}_k] = \mathfrak{H}_k,
\end{align*}
where
\begin{align*}
\mathfrak{H}_k[f] = |k| \textup{p.v.} \int_{-1}^1 \frac{H_k(y,y')}{2i(y-y')} f(y') \dee y', 
\end{align*}
where $H_k$ is a continuous function given by the formula
\begin{align}
H_k(y,y') := -\frac{\sinh(k(y+y'))}{\sinh(2k)}
\end{align}
\end{lemma}
\begin{proof}
Consider the regularizations
\begin{align*}
\frac{1}{|k|}\mathfrak{J}_k^\eps[f](y) := \int_{-1}^1 \frac{G_k(y,y') (y-y')}{(y-y')^2 + \eps^2} f(y') \dee y'. 
\end{align*}
Integrating by parts (and using the boundary conditions), we have
\begin{align*}
\frac{1}{|k|} (\mathfrak{J}_k^\eps[\partial_y f] - \partial_y \mathfrak{J}_k^\eps[f])
= \int_{-1}^1 \frac{(y-y')}{(y-y')^2 + \eps^2}\left(-\partial_{y'}G_k(y,y') - \partial_{y}G_k(y,y')\right) f(y') \dee y'. 
\end{align*}
The right-hand side passes to the limit in $L^2$ to $\mathfrak{H}_k$. If $\partial_y f \in L^2$ then $\mathfrak{J}_k^\eps[\partial_y f]$ passes to the limit in $L^2$ as well.
It follows that $\mathfrak{J}_k^\eps[f] \in H^1$ and $\partial_y \mathfrak{J}_k^\eps[f]$ is given by the above formula. First of all, a direct computation yields
\begin{align*}
&\pa_y G_k(y,y')\\
=&-\frac{1}{\sinh(2k)}\left\{\begin{array}{cc}
\frac{\sinh(2k)}{2}+\underbrace{\frac{1}{4}e^{2k}(e^{k(y-y')}-1)+\frac{1}{4}e^{-2k}(1-e^{-k(y-y')})}_{=:H(y-y')}-\frac{1}{2}\sinh(k(y'+y)),&\quad y\leq y';\\
-\frac{\sinh(2k)}{2}+\underbrace{\frac{1}{4}e^{-2k}(e^{k(y-y')}-1)+\frac{1}{4}e^{2k}(1-e^{-k(y-y')})}_{=:H(y-y')}-\frac{1}{2}\sinh(k(y'+y)),&\quad y\geq y'.\end{array}\right.
\end{align*}
We now observe that we can rewrite the kernel $\pa_y G_k(y,y')$ using the notations $y-y'=\eta,\, y+y'=\zeta$ as follows:
\begin{align}\n
&\pa_y G_k(y,y')\\ \n
=&-\frac{1}{\sinh(2k)}\left(-\frac{1}{2}\sinh(2k){\rm sign}(\eta)+\left(\frac{1}{2}\sinh \big((-2+|\eta|)k\big)+\frac{1}{2}\sinh(2k)\right)\frac{\eta}{|\eta|}-\frac{1}{2}\sinh(k\zeta)\right)\\
=:&J(\eta)+H(\eta)+S(\zeta).\label{JHS} 
\end{align}
Here `$J$' is the jump part, `$H$' is the $C^1$ part and `$S$' is the smooth part. We notice that functions of the $\eta$ variable are annihilated by $- \p_y - \p_{y'}$, and therefore only the $``S"$ term contributes to the commutator. Indeed, we have that
\begin{align}
-\pa_y G_k(y,y')-\pa_{y'}G_k(y,y')=-\frac{\sinh(k(y+y'))}{\sinh(2k)}.
\end{align}
\end{proof}

We have the following commutator estimate on $\mathfrak{H}_k$.
\begin{lemma}\label{lem:H}
There holds the estimate
\begin{align} \label{frakH}
\norm{\mathfrak{H}_k}_{L^2 \to L^2} \lesssim |k|. 
\end{align}
\end{lemma}
\begin{proof} We decompose in a similar manner to Lemma \ref{lem:BoundI}, as follows: 
\begin{align*}
\frac{1}{|k|}\mathfrak{H}_{k}f  = &  \textup{p.v.}  \int_{-1}^1 H_k(y,y') \frac{f(y')}{2i(y-y')} dy' \\   
= & \frac{k}{\abs{k}} \textup{p.v.}  \int_{[-1,1] - (y - \frac{1}{k}, y + \frac{1}{k})} H_k(y,y') \frac{f(y')}{2i(y-y')} dy' +  \frac{k}{\abs{k}} \textup{p.v.}  \int_{y - \frac{1}{k}}^{y + \frac{1}{k}} H_k(y,y') \frac{f(y')}{2i(y-y')} dy' \\
 = &\frac{k}{\abs{k}} \textup{p.v.}  \int_{[-1,1] - (y - \frac{1}{k}, y + \frac{1}{k})} H_k(y,y') \frac{f(y')}{2i(y-y')} dy'  +  \frac{k}{\abs{k}} H_k(y,y) \textup{p.v.}  \int_{y - \frac{1}{k}}^{y + \frac{1}{k}} \frac{f(y')}{2i(y-y')} dy' \\
 & +   \frac{k}{\abs{k}}  \textup{p.v.} \int_{y - \frac{1}{k}}^{y + \frac{1}{k}} [H_k(y,y') - H_k(y,y)]  \frac{f(y')}{2i(y-y')} dy' \\
 = & \mathcal{T}_1 + \mathcal{T}_2 + \mathcal{T}_3. 
\end{align*}
We will now provide bounds on $\mathcal{T}_1, \mathcal{T}_2, \mathcal{T}_3$ successively. 

\vspace{2 mm}

\noindent \textit{Bounds on $\mathcal{T}_1$:} We again split into $\mathcal{T}_1 = \mathcal{T}_{1,\ge} + \mathcal{T}_{1,\le}$, where we define 
\begin{align}
\mathcal{T}_{1,\ge} f := & \frac{k}{\abs{k}} \textup{p.v.}  \int_{y' = y + \frac{1}{k}}^1 H_k(y,y') \frac{f(y')}{2i(y-y')} dy', \\
\mathcal{T}_{1,\le} f := & \frac{k}{\abs{k}} \textup{p.v.}  \int_{y' = -1}^{y - \frac{1}{k}} H_k(y,y') \frac{f(y')}{2i(y-y')} dy'.
\end{align}
By symmetry it suffices to estimate $\mathcal{T}_{1,\ge}$. For this, we have the kernel 
\begin{align}
\bold{K}(y, y') := \mathbbm{1}(y' \ge y + \frac1k) H_k(y,y') \frac{1}{2i(y' - y)}.
\end{align}
We will first fix a $y$ and compute the $L^1_{y'}$ norm. Indeed, we have 
\begin{align}\n
\| \bold{K}(y, \cdot) \|_{L^1_{y'}} \lesssim & \int_{y' = y + \frac1k}^1 \frac{\sinh(k(y+y'))}{(y' - y) \sinh(2k)} dy' \lesssim \int_{y' = y + \frac1k}^1 \frac{e^{k(y+y')}}{(y' - y) e^{2k}} dy' \\
\lesssim & e^{-2k(1-y)} \int_1^{k(1-y)} \frac{e^u}{u} du \lesssim 1. 
\end{align}
Taking now the supremum in $y$, we obtain $\| \bold{K} \|_{L^\infty_y L^1_{y'}}  \lesssim 1$. We now fix a $y'$ in the range $1/k \le y' \le 1$. Then we compute 
\begin{align} \n
\| \bold{K}(\cdot, y') \|_{L^1_y} \lesssim & \int_{y = 0}^{y' - \frac{1}{k}} \frac{\sinh(k(y+y'))}{(y' - y) \sinh(2k)} dy \lesssim \int_{y = 0}^{y' - \frac{1}{k}} \frac{e^{k(y+y')}}{(y' - y) e^{2k}} dy \\
\lesssim & e^{-2k(1-y')} \int_1^{y' k} \frac{e^{-u}}{u} du \lesssim e^{2k(1-y')}. 
\end{align}
Taking now the supremum in $y'$ over the range $1/k \le y' \le 1$, we conclude $\| \bold{K} \|_{L^\infty_{y'} L^1_{y}}  \lesssim 1$. By applying the analogous argument also to $\mathcal{T}_{1,\le}$, we have 
\begin{align} \label{T1bd:hk}
\| \mathcal{T}_{1} f \|_{L^2} \lesssim \| f \|_{L^2}.
\end{align}

\vspace{2 mm}

\noindent \textit{Bounds on $\mathcal{T}_2$:} For the operator $\mathcal{T}_2$, we have 
\begin{align} \n
\| \mathcal{T}_2f\|_{L^2} = & \| \frac{k}{\abs{k}} H_k(y,y) \textup{p.v.}  \int_{y - \frac{1}{k}}^{y + \frac{1}{k}} \frac{f(y')}{2i(y-y')} dy' \|_{L^2} \lesssim \| H_k(y, y) \|_{L^\infty_{y}} \|  \textup{p.v.}  \int_{y - \frac{1}{k}}^{y + \frac{1}{k}} \frac{f(y')}{2i(y-y')} dy' \|_{L^2} \\ \label{htrn:1}
\lesssim & \| \textup{p.v.}  \int_{y - \frac{1}{k}}^{y + \frac{1}{k}} \frac{f(y')}{2i(y-y')} dy' \|_{L^2}.
\end{align}
Above, we have used the bound 
\begin{align*}
\| H_k(y,y) \|_{L^\infty_y}\lesssim \| \frac{\sinh(2ky)}{\sinh(2k)} \|_{L^\infty_y} \lesssim 1. 
\end{align*}
From here, the desired bound follows from the corresponding boundedness of the truncated Hilbert transform, just as in Lemma \ref{lem:BoundI}.
\vspace{2 mm}

\noindent \textit{Bounds on $\mathcal{T}_3$:}  For the operator $\mathcal{T}_3$, we use Schur's Test again. In this case, we have 
\begin{align*}
\int_{y - \frac{1}{k}}^{y + \frac{1}{k}} |k| \frac{|H_k(y, y') - H_k(y,y)|}{k|y - y'|} dy' \lesssim \sup_{y, y'} \| \frac{H_k(y,y') - H_k(y,y)}{k|y - y'|} \Big| \int_{y - \frac{1}{k}}^{y + \frac1k}  |k| dy' \lesssim 1, \\
\int_{y' - \frac{1}{k}}^{y' + \frac{1}{k}}  |k| \frac{|H_k(y, y') - H_k(y,y)|}{k|y - y'|} dy \lesssim \sup_{y, y'} \| \frac{H_k(y,y') - H_k(y,y)}{k|y - y'|} \Big| \int_{y' - \frac{1}{k}}^{y' + \frac1k}  |k| dy \lesssim 1. 
\end{align*}  
Therefore, again by applying Schur's test, we have 
\begin{align} \label{T3bd}
\| \mathcal{T}_3f\|_{L^2}  \lesssim &  \| f \|_{L^2}.
\end{align}
Bringing the bounds on $\mathcal{T}_1, \mathcal{T}_2, \mathcal{T}_3$ together finishes the proof of the lemma. 
\end{proof}

Finally we point out the following symmetry properties. 
\begin{lemma} \label{lem:AS}
For all $f,g \in L^2$ there holds 
\begin{align*}
\overline{\mathfrak{J}_k[f]} = - \mathfrak{J}_k[\overline{f}]
\end{align*}
and
\begin{align}
\int_{-1}^1 \bar{f} \mathfrak{J}_k[g] dy = - \int_{-1}^1 \mathfrak{J}_k[\bar{f}] g  dy, 
\end{align}
which in particular, implies $\mathfrak{J}_k = \mathfrak{J}^\ast_k$. 
\end{lemma}
\begin{proof}
The first identity follows by definition. 
For the latter, we have 
\begin{align*}
\int_{-1}^1 \bar{f} \siming{\mathfrak{J}_k[g]} dy = & \int_{-1}^1 \bar{f} \lim_{\eps \rightarrow 0} \siming{\mathfrak{J}_k^{\eps}}[g] dy = \lim_{\eps \rightarrow 0} \int_{-1}^1 \bar{f} \siming{\mathfrak{J}_k^\eps}[g] dy \\
= & \lim_{\eps \rightarrow 0} \int_{-1}^1 \bar{f}(y) \Big( \int_{|y' - y| \ge \eps} \frac{G_k(y,y')}{2i(y - y')} g(y') dy' \Big) dy \\
= & \lim_{\eps \rightarrow 0}  \int_{-1}^1  g(y') \Big( \int_{|y - y'| \ge \eps}  \frac{G_k(y,y')}{2i(y - y')} \bar{f}(y) dy\Big) dy' \\
= & -  \lim_{\eps \rightarrow 0} \int_{-1}^1  g(y') \Big( \int_{|y - y'| \ge \eps}  \frac{G_k(y', y)}{2i(y'- y)} \bar{f}(y) dy\Big) dy' \\
= & - \int g \siming{\mathfrak{J}_k}[\bar{f}] dy', 
\end{align*}
which is the desired result. 
\end{proof}

%% file: edlittle.tex
In this section we prove Proposition \ref{prop:lin}, which specifically concerns the linearized problem \eqref{NSchannel}. 
To this end, we begin by computing the time derivative of the energy \eqref{def:Ek}
\begin{align}\n
\frac{1}{2}\frac{d}{dt}E_k[\omega_k] = & \frac{1}{2} \left(\frac{d}{dt}\| \omega_{k} \|_{L^2}^2   + c_{\alpha} \nu^{\frac23} |k|^{-\frac23}  \frac{d}{dt}\| \p_y \omega_{k} \|_{L^2}^2 + c_{\beta}{\nu^{\frac{1}{3}}} |k|^{-\frac43}  \frac{d}{dt} \Re \langle ik \omega_{k}, \p_y \omega_{k} \rangle \right. \\ \n
&\left. + c_{\tau}  \frac{d}{dt} \Re \langle \omega_k, \mathfrak{J}_k[\omega_k] \rangle + c_{\tau}c_\alpha \nu^{\frac{2}{3}} \abs{k}^{-\frac{2}{3}}  \frac{d}{dt} \Re \langle \partial_y \omega_k, \mathfrak{J}_k[\partial_y \omega_k] \rangle \right) \\ 
 =: &\bb T_\gamma + \bb T_\alpha + \bb T_\beta + \bb T_{\tau} + \bb T_{\tau \alpha}.\label{T_all}
\end{align}
Each term is confronted in the following individual subsections.
Since this entire section is $k$-by-$k$, we omit the $k$'s whenever it is clear from context. Without loss of generality, we assume $k> 0$ in this section.

\subsection{Estimate of the $\mathbb{T}_\gamma$ Terms}
The basic $L^2$ energy estimate is almost immediate. 
\begin{lemma}[$\bb T_\gamma$ Estimate] \label{lem:LinGamma}
Under the hypotheses of Proposition \ref{prop:lin}, we have the following 
\begin{align}
\frac{1}{2} \frac{d}{dt} \| \omega_k \|_{L^2}^2 + \nu \| \grad_k \omega_k \|_{L^2}^2 \lesssim \delta_0 \abs{k}^{\mh{\delta}-1}  \mathrm{D}_{k,\tau}.%
\end{align}
Here the parameter $\delta_0$ is defined in \eqref{asmp} and the diffusion term $\mathrm{D}_{k,\tau}$ is defined in \eqref{def:Dks}. The implicit constant depends on $\|U \|_{C^3}.$
\end{lemma}
\begin{proof}
By direct calculation we have  
\begin{align*}
\frac{1}{2} \frac{d}{dt}\| \omega_k \|_{L^2}^2 + \nu \| \grad_k \omega_k \|_{L^2}^2 =  \Re\ ik \langle U'' \phi_k, \omega_k \rangle . 
\end{align*}
The latter term is estimated as follows using integration by parts, Cauchy-Schwarz, and the smallness assumption \eqref{asmp}, 
\begin{align*}
| \Re\ ik \langle U'' \phi_k,  \omega_k \rangle| = & | \Re\ ik \langle U'' \phi_k,  \Delta_k \phi_k \rangle|  
\lesssim     | \Re\ ik \langle U'' \nabla_k \phi_k,  \nabla_k \phi_k \rangle| +  | \Re\ ik \langle U''' \phi_k,  \grad_k \phi_k \rangle| \\
 \lesssim&  \| U' \|_{C^2} |k| \| \nabla_k \phi_k \|_{L^2}^2\siming{ \lesssim \|W\|_{H^4}|k|^{-1}\mathrm{D}_{k,\tau}} 
 \lesssim  \delta_0 |k|^{-1} \mathrm{D}_{k,\tau}, 
\end{align*}
which completes the proof. 
\end{proof}

\subsection{Estimate of the $\bb T_\alpha$ Terms}
The following is also a relatively straightforward calculation.  
\begin{lemma}[$\bb T_\alpha$ Estimate] \label{lem:LinAlpha}
Under the hypotheses of Proposition \ref{prop:lin}, we have the following 
\begin{align}
\frac{1}{2}\frac{d}{dt} \nu^{2/3}|k|^{-2/3} \norm{\partial_y \omega_k}_{L^2}^2 + \nu (\nu^{2/3}|k|^{-2/3}) \norm{\grad_k \partial_y \omega_k}_{L^2}^2 \lesssim \delta_0 \mathrm{D}_{k,\gamma} + \delta_0 \nu^{1/3} |k|^{2/3} \norm{\omega_k}_{L^2}^2. 
\end{align}
Here $\delta_0,\, \mathrm{D}_{k,\gamma}$ are defined in \eqref{asmp} and \eqref{def:Dks}. 
The implicit constant depends on $\|U \|_{ C^3}$.
\end{lemma}
\begin{proof}
We compute the time-derivative
\begin{align*}
\frac{1}{2}\frac{d}{dt}\norm{\partial_y \omega_k}_{L^2}^2 =  \nu \Re \brak{\Delta_k \partial_y \omega_k, \partial_y \omega_k}-\Re \brak{ikU' \omega_k, \partial_y \omega_k} + \Re \brak{\partial_y (U'' ik \phi_k),\partial_y \omega_k}. 
\end{align*}
Using that $\Delta_k \omega_k = \partial_{yy} \omega_k = 0$ on the boundary $\{y=\pm 1\}$, we may still integrate by parts on the first term, which gives rise to the desired dissipation term, 
\begin{align*}
\nu \Re \brak{\Delta_k \partial_y \omega_k, \partial_y \omega_k} = - \nu \norm{\grad_k \partial_y \omega_k}_{L^2}^2. 
\end{align*}
The second and third terms are handled via Cauchy-Schwarz, elliptic regularity and the smallness assumption \eqref{asmp}, 
\begin{align*}
\abs{\Re \brak{ikU' \omega_k, \partial_y \omega_k}} \lesssim \norm{W}_{H^4}&|k|\norm{\omega_k}_{L^2}\norm{\pa_y \omega_k}_{L^2}\lesssim \delta_0 \abs{k} \norm{ \omega_k}_{L^2} \norm{\partial_y \omega_k}_{L^2}, \\
\lf|\Re \brak{\partial_y (U'' ik \phi_k),\partial_y \omega_k}\rg| & \lesssim \delta_0 \norm{ \omega_k}_{L^2} \norm{\partial_y \omega_k}_{L^2}, 
\end{align*}
which completes the desired estimates after noting that
\begin{align*}
\frac{\nu^{2/3}}{|k|^{2/3}} \abs{k} \norm{ \omega_k}_{L^2} \norm{\partial_y \omega_k}_{L^2} \lesssim \nu \norm{\grad_k \omega_k}_{L^2}^2 + \nu^{1/3} |k|^{2/3}\norm{\omega_k}_{L^2}^2\approx\mathrm{D}_{k,\gamma}+\nu^{1/3} |k|^{2/3}\norm{\omega_k}_{L^2}^2. 
\end{align*}
\end{proof} 

\subsection{Estimate of the $\bb T_\beta$ Terms}
As in the case with standard hypocoercivity approaches, the cross-term $\bb T_\beta$, produces the enhanced dissipation.
\begin{lemma} \label{lem:LinBeta}
There holds the following, 
\begin{align*}
-\frac{\nu^{1/3}}{|k|^{4/3}}\frac{d}{dt} \Re \langle ik \omega_{k}, \p_y \omega_{k} \rangle + \nu^{1/3} \abs{k}^{2/3} \norm{\sqrt{U'} \omega_k}_{L^2}^2 & \lesssim  \mathrm{D}_{k,\gamma}^{1/2}\mathrm{D}_{k,\alpha}^{1/2}. 
\end{align*}
Here, the parameter $\delta_0$ and the dissipation $\{\mathrm{D}_{k,\gamma}, \mathrm{D}_{k,\al}\}$ are defined in \eqref{asmp}, \eqref{def:Dks}. The implicit constant depends on $\|U\|_{C^3}.$
\end{lemma} 
\begin{proof} 
Here we have by integration by parts (using that $\omega_k$, $\phi_k$ vanish on the boundary), 
\begin{align*}
\frac{d}{dt} \Re \langle ik \omega_{k}, \p_y \omega_{k} \rangle & = \Re \langle ik ( - ik U \omega_k {+ U''} ik \phi_k + \nu \Delta_k \omega_k), \p_y \omega_k \rangle \\
& \quad + \Re \langle ik \omega_k, \partial_y ( - ik U \omega_k {+ U''} ik \phi_k + \nu \Delta_k \omega_k) \rangle \\
& = \abs{k}^2 \norm{\sqrt{U'} \omega_k}_{L^2}^2  {-}2\Re \abs{k}^2 \brak{U''  \phi_k,\partial_y \omega_k} \\
& \quad + \mathrm{Im}\ k \brak{\omega_k,\nu \partial_y\Delta_k \omega_k} + \mathrm{Im}\ k \brak{\partial_y \omega_k,\nu \Delta_k \omega_k}. 
\end{align*}The first term on the right hand side is the dissipation and the remaining terms are error terms. 
To treat the first error term, we integrate by parts (using the boundary conditions) and elliptic regularity to obtain 
\begin{align*}
\abs{ 2\Re \abs{k}^2 \brak{U''  \phi_k,\partial_y \omega_k}} & \lesssim \left(\norm{U''}_{L^\infty} \norm{k\partial_y \phi_k}_{L^2} + \norm{U'''}_{L^\infty} \norm{k \phi_k}_{L^2} \right) \abs{k}\norm{\omega_k}_{L^2} \\
& \lesssim \delta_0 \abs{k}\norm{\omega_k}_{L^2}^2. 
\end{align*}
For the latter error terms note that by integrating by parts and Cauchy-Schwarz,
\begin{align*}
\abs{\mathrm{Im}\ k \brak{\omega_k,\nu \partial_y\Delta_k \omega_k} + \mathrm{Im} \ k \brak{\partial_y \omega_k,\nu \Delta_k \omega_k}} & \lesssim \nu\abs{k}\norm{\partial_y \omega_k}_{L^2}\left(|k|^2 \norm{\omega_k}_{L^2} + \norm{\partial_{yy}\omega_k}_{L^2} \right), 
\end{align*}
which completes the desired estimate upon noting that
\begin{align*}
\nu^{4/3} \abs{k}^{-1/3} \norm{\partial_y \omega_k}_{L^2}\left(|k|^2 \norm{\omega_k}_{L^2} + \norm{\partial_{yy}\omega_k}_{L^2} \right) \lesssim \left(\nu \norm{\grad_k \omega_k}_{L^2}\right)^{1/2} \left(\nu^{5/3} |k|^{-2/3} \norm{\grad_k \partial_y \omega_k}_{L^2}\right)^{1/2}.
\end{align*}
\end{proof} 

\subsection{Estimates of the $\bb T_\tau$ and $\bb T_{\tau\alpha}$ Terms} \label{sec:LinTau}
The most interesting contribution relative to existing works on hypocoercivity is that of $\bb T_\tau$ (and $\bb T_{\tau \alpha}$). 
First, we prove the main result for $\bb T_\tau$ and below we provide the estimate on $\bb T_{\tau \alpha}$. 
\begin{lemma} \label{lem:LinTau}
Under the regularity hypotheses on $W$ \eqref{asmp}, for all $\delta_0>0$ sufficiently small (depending only on universal constants), the following estimate holds 
\begin{align*}
\frac{d}{dt} \Re \langle \omega_k, \mathfrak{J}_k[\omega_k] \rangle + \frac{1}{8}\mathrm{D}_{k,\tau} \lesssim  \mathrm{D}_{k,\gamma}. 
\end{align*}
The implicit constant depends on $\|U\|_{ C^3}.$
\end{lemma}
\begin{remark}
In the case of the Euler equations ($\nu = 0$), one simply obtains
\begin{align*}
\frac{d}{dt} \Re \langle \omega_k, \mathfrak{J}_k[\omega_k] \rangle + \frac{1}{8} \mathrm{D}_{k,\tau} \leq 0, 
\end{align*}
proving the $L^2_t H^1_x L^2_y$ inviscid damping estimate for $L^2$ initial data. 
\end{remark}
\begin{proof} 
First of all, since the SIO $\mathfrak{J}_k$ is \siming{symmetric} 
(Lemma \ref{lem:AS}), we obtain from the equation \eqref{NSchannel} that
\begin{align}\label{LinTau_1}
\frac{1}{2}\frac{d}{dt} \Re \langle \omega_k, \mathfrak{J}_k[\omega_k] \rangle & =  \Re \brak{ -Uik \omega_k \myr{+} U'' ik \phi_k + \nu\Delta_k \omega_k, \mathfrak{J}_k[\omega_k]}  
  =: T_1 + T_2 + T_3. 
\end{align}
The most straightforward term is that arising from the dissipation, $T_3$. 
Using the symmetry Lemma \ref{lem:AS} and the $[\pa_y,\mathfrak{J}_k]$-commutator estimate (Lemma \ref{lem:CommIpy}, Lemma \ref{lem:H}), we have (using that $\omega_k$ vanishes on the boundary and that $\mathfrak{J}_k$ is bounded), 
\ifx\begin{align*}
T_3 & = \nu \Re \int_{-1}^1 \Delta_k \overline{\omega_k} \mathfrak{J}_k[\omega_k] \dy = \nu \Re \int_{-1}^1 \overline{\mathfrak{J}_k[\Delta_k\omega_k]} \omega_k \dy = -\nu \Re \int \overline{\grad_k \omega_k}  \mathfrak{J}_k[\grad_k \omega_k] \dee y \lesssim \nu \norm{\grad_k \omega_k}_{L^2}^2 \lesssim \mathcal{D}_\gamma,
\end{align*}\fi
\siming{
\begin{align*}
T_3 & = \nu \Re \int_{-1}^1 \Delta_k \omega_k \overline{\mathfrak{J}_k[\omega_k]} \dy = \nu \Re \int_{-1}^1 \mathfrak{J}_k[\Delta_k\omega_k] \overline{\omega_k} \dy\\
& = \nu \Re \int\lf( -\mathfrak{J}_k[\grad_k \omega_k ]\overline{\grad_k \omega_k}+[\mathfrak{J}_k,\pa_y](\pa_y \omega_k)\overline{\omega_k}\rg) \dee y \\
&\leq \nu\lf(\norm{\mathfrak{J}_k}_{L^2\rightarrow L^2}+|k|^{-1}\norm{[\mathfrak{J}_k,\pa_y]}_{L^2\rightarrow L^2}\rg)\norm{\na_k\omega_k}_{L^2}^2\lesssim \nu \norm{\grad_k \omega_k}_{L^2}^2 \lesssim \mathrm{D}_{k,\gamma},
\end{align*}}
which is our desired estimate on $T_3$ in \eqref{LinTau_1}.

Next, we estimate $T_2$.
By Lemma \ref{lem:CommIpy}, \ref{lem:H}, and $\omega_k = \Delta_k \phi_k$, we have \ifx
\begin{align*}
T_2 
& = \Re\int_{-1}^1  ik \grad_k \left(U''(y) \overline{\phi}_k\right) \cdot \mathfrak{I}[\grad_k \phi_k] dy \\
& \lesssim \abs{k} \norm{U''}_{C^1} \left(\norm{\phi_k}_{L^2} + \norm{\grad_k \phi}_{L^2}\right) \norm{\mathfrak{I}[\grad_k \phi_k]}_{L^2} \\
& \lesssim \abs{k} \norm{U''}_{C^1} \norm{\grad_k \phi_k}_{L^2}^2 \\
& \lesssim \delta_0 \mathcal{D}_\tau. 
\end{align*}\fi
\siming{\begin{align*}
T_2 & = \Re\int_{-1}^1  \left(ik U'' \phi_k\right) \overline{\grad_k\cdot \mathfrak{J}_k[\grad_k \phi_k]+[\mathfrak{J}_k,\pa_y][\pa_y \phi_k]} dy \\
&=\Re\int_{-1}^1   -\na_k\left(ik U'' \phi_k\right)\cdot \overline{ \mathfrak{J}_k[\grad_k \phi_k]} +\left(ik U'' \phi_k\right)\overline{[\mathfrak{J}_k,\pa_y][\pa_y \phi_k]}dy \\
& \lesssim \abs{k} \norm{U''}_{C^1} \lf( \|\mathfrak{J}_k\|_{L^2\rightarrow L^2}+|k|^{-1}\norm{[\mathfrak{J}_k,\pa_y]}_{L^2\rightarrow L^2}\rg)\norm{\grad_k \phi_k}_{L^2}^2\lesssim \delta_0 \mathrm{D}_{k,\tau}. 
\end{align*}}
where the last line used Lemma \ref{lem:BoundI}. 
This is sufficient to estimate $T_2$ in \eqref{LinTau_1}. 

Finally we turn to the most interesting contribution, $T_1$ which will lead to the negative-definite term.
By Taylor's theorem, 
\begin{align*}
U(y)-U(y')=U'(y)(y'-y) + \frac{1}{2}U''(y)(y'-y)^2 + \frac{1}{6}\int_y^{y'}U'''(s)(y'-s)^2 \dee s.
\end{align*}
Note that remainder satisfies 
\begin{subequations} \label{r_fn}
\begin{align}
&r(y,y'):=\frac{1}{6(y-y')}\int_y^{y'}U'''(s)(y'-s)^2 \dee s \in L^\infty, \\
& \abs{r(y,y')} \lesssim \norm{U'''}_{L^\infty} \abs{y-y'}^2, \\ 
&\abs{\partial_y r(y,y')}  + \abs{\partial_{y'} r(y,y')} \lesssim \norm{U'''}_{L^\infty} \abs{y-y'}, \\
&\abs{\partial_{y'} \partial_y r(y,y')} \lesssim \norm{U'''}_{L^\infty}. 
\end{align} 
\end{subequations}
Next, by the definition of $\mathfrak{J}_k$ and Lemma \ref{lem:AS}, we have 
\begin{align}\n
T_1  =& -\frac{k}{2}\Re\lf(\brak{iU\omega_k,\mathfrak{J}_k\omega_k}+\brak{\mathfrak{J}_k[i U\omega_k],\omega_k}
\rg)\\
=&\siming{\frac{|k|^2}{4}}\Re \int_{-1}^1\siming{\text{p.v.}}\lf(\int_{-1}^1   U(y)\omega_k(y)\frac{1}{(y-y')} G_k(y,y')\overline{\omega_k(y')}dy'\rg)dy\n \\
&-\siming{\frac{|k|^2}{4}}\Re \int_{-1}^1\siming{\text{p.v.}}\lf(\overline{\int_{-1}^1 \overline{\omega_k(y)}\frac{1}{(y-y')} G_k(y,y') U(y')\omega_k(y')dy'}\rg) dy\n \\
=&\frac{\abs{k}^{2\mh{-\delta}}}{4}\Re \int_{-1}^1\int_{-1}^1 \overline{\omega_k(y)}\frac{(U(y)-U(y'))}{(y-y')} G_k(y,y')\omega_k(y')dy'dy\n \\
=&\frac{\abs{k}^{2\mh{-\delta}}}{4}\Re \int_{-1}^1\int_{-1}^1 \overline{\omega_k(y)}\frac{U'(y) (y-y')}{(y-y')} G_k(y,y')\omega_k(y')dy'dy\n \\
& \quad +\frac{\abs{k}^{2\mh{-\delta}}}{8 }\Re \int_{-1}^1\int_{-1}^1 \overline{\omega_k(y)}  U''(y)(y-y')  G_k(y,y')\omega_k(y')dy'dy\n \\
& \quad +\frac{\abs{k}^{2\mh{-\delta}}}{4}\Re \int_{-1}^1\int_{-1}^1 \overline{\omega_k(y)} r(y,y') G_k(y,y')\omega_k(y')dy'dy\n \\
=: & T_{11} + T_{12} + T_{13}.\label{LinTau_2}
\end{align}
The term $T_{11}$ leads to $-\mathrm{D}_{k,\tau}$.
Indeed, applying the relation $\de_k\phi_k=\omega_k$ and integrating by parts (using the boundary condition that $\phi_k$ vanishes on the boundary), 
\begin{align*}
T_{11}
 &=\frac{\abs{k}^{2\mh{-\delta}}}{4}\Re\int_{-1}^1 U'\overline{\Delta_k\phi_k} \phi_k \dy = - \frac{\abs{k}^{2\mh{-\delta}}}{4} \norm{\sqrt{U'} \grad_k \phi_k}_{L^2}^2  -\frac{\abs{k}^{2\mh{-\delta}}}{4}\Re \int U'' \pa_y\phi_k\overline{\phi_k}\dy \\
& \leq - \frac{\abs{k}^{2\mh{-\delta}}}{4} \norm{\sqrt{U'} \grad_k \phi_k}_{L^2}^2  +\frac{\abs{k}^{2\mh{-\delta}}}{4} \norm{U''}_{L^\infty} \norm{\pa_y\phi_k}_{L^2} \norm{\phi_k}_{L^2}. 
\end{align*}
Therefore, for $\|W_{in}\|_{H^4}\leq\delta_0$  sufficiently small depending only on universal constants, there holds 
\begin{align}\label{LinTau2T1}
T_{11} \leq -\frac{\abs{k}^{2\mh{-\delta}}}{6} \norm{\grad_k \phi_k}_{L^2}^2. 
\end{align}
Next, consider $T_{12}$, which we begin by making the replacement $\omega_k = \Delta_k \phi_k$ and integrating by parts in both $y$ and $y'$ (noting that the Green's function vanishes on the boundary in both variables, i.e. $G_k(\pm1,y') = G_k(y,\pm 1)\siming{=\pa_yG_k(y,\pm 1)=\pa_{y'}G_k(\pm1,y') = 0 }$) to obtain 
\begin{align*}
T_{12} & = \frac{1}{8 }\abs{k}^{2\mh{-\delta}} \Re \int_{-1}^1\int_{-1}^1 \overline{\Delta_k \phi_k(y)} U''(y) (y-y') G_k(y,y') \Delta_k \phi_k(y')\dy' \dy \\ 
 &= \frac{1}{8 }\abs{k}^{2\mh{-\delta}} \Re \int_{-1}^1\int_{-1}^1 \overline{ \na_{k,y}\phi_k(y)}\cdot\na_{k,y}\bigg (  \na_{k,y'}\left(U''(y) (y-y') G_k(y,y')\right) \cdot\na_{k,y'}\phi_k(y')\bigg)\dy' \dy.
\end{align*}
We will return to this estimate in a moment.
We observe the following estimates on the integral kernels $\mathcal{N} = U''(y) (y-y') G_k(y,y')$
\begin{align*}
\norm{\mathcal{N}_k}_{L^2_{y,y'}} & \lesssim \norm{U''}_{L^\infty} \abs{k}^{-2}, \\
\norm{\partial_y \mathcal{N}_k}_{L^2_{y,y'}} + \norm{\partial_{y'} \mathcal{N}_k}_{L^2_{y,y'}} & \lesssim \norm{U'''}_{L^\infty} \abs{k}^{-1} ,\\
\norm{\partial_{y} \partial_{y'} \mathcal{N}_k }_{L^2_{y,y'}} & \lesssim \norm{U'''}_{L^\infty},
\end{align*}
which follow from the the formula for $G_k$, \eqref{G_k}, and direct calculation.
Therefore, by Cauchy-Schwarz, we have
\begin{align}
T_{12} \lesssim \delta_0\abs{k}^{2\mh{-\delta}} \norm{\grad_k \phi_k}_{L^2}^2,\label{LinTau2T2}
\end{align}
which suffices to treat this term. 
We similarly set up $T_{13}$, which we treat in essentially the same manner: 
\begin{align*}
T_{13}  
&\siming{=-\Re \frac{\abs{k}^{2}}{4}\int_{-1}^1\int_{-1}^1 \overline{\grad_{k,y} \phi_k(y)} \cdot\grad_{k,y} \lf(\grad_{k,y'}\left(r(y,y') G_k(y,y') \right)\cdot \grad_{k,y'}\phi_k(y')\rg) \dy' \dy,}
\end{align*}
where in order to integrate by parts twice 
we needed to use that $rG_k$ does not have a singularity in the second derivative at $y=y'$ (unlike $G_k$). 
Next we record the relevant estimates on $rG_k$. Recall that we assume $k> 0$ without loss of generality in this section.   
First, note that by \eqref{G_k} and \eqref{r_fn},
\begin{align*}
\abs{rG_k} & \lesssim \delta_0 \frac{\abs{y-y'}^2}{|k|\sinh(2k)}\left\{\begin{array}{cc}\sinh(k(1-y'))\sinh(k(1+y)),&\quad y\leq y';\\
\sinh(k(1-y))\sinh(k(1+y')),&\quad y\geq y'.\end{array}\right. \\
\abs{\partial_y(rG_k)} & \lesssim \delta_0 \frac{\abs{y-y'}}{|k|\sinh(2k)}\left\{\begin{array}{cc}\sinh(k(1-y'))\sinh(k(1+y)),&\quad y\leq y';\\
\sinh(k(1-y))\sinh(k(1+y')),&\quad y\geq y'.\end{array}\right. \\
& \quad + \delta_0 \frac{\abs{y-y'}^2}{\sinh(2k)}\left\{\begin{array}{cc}\sinh(k(1-y'))\cosh(k(1+y)),&\quad y\leq y';\\
|-\cosh(k(1-y))\sinh(k(1+y'))|,&\quad y\geq y'.\end{array}\right. \\
\abs{\partial_{y'}(rG_k)} & \lesssim \delta_0 \frac{\abs{y-y'}}{|k|\sinh(2k)}\left\{\begin{array}{cc}\sinh(k(1-y'))\sinh(k(1+y)),&\quad y\leq y';\\
\sinh(k(1-y))\sinh(k(1+y')),&\quad y\geq y'.\end{array}\right. \\
& \quad + \delta_0 \frac{\abs{y-y'}^2}{\sinh(2k)}\left\{\begin{array}{cc}|-\cosh(k(1-y'))\sinh(k(1+y))|,&\quad y\leq y';\\
\sinh(k(1-y))\cosh(k(1+y')),&\quad y\geq y'.\end{array}\right. \\
\abs{\partial_y \partial_{y'} (rG_k)} & \lesssim \delta_0 \frac{1}{|k|\sinh(2k)}\left\{\begin{array}{cc}\sinh(k(1-y'))\sinh(k(1+y)),&\quad y\leq y';\\
\sinh(k(1-y))\sinh(k(1+y')),&\quad y\geq y'.\end{array}\right. \\
& \quad + \delta_0 \frac{\abs{y-y'}}{\sinh(2k)}\left\{\begin{array}{cc}|-\cosh(k(1-y'))\sinh(k(1+y))|,&\quad y\leq y';\\
\sinh(k(1-y))\cosh(k(1+y'))|,&\quad y\geq y'.\end{array}\right. \\
& \quad + \delta_0\frac{\abs{y-y'}}{\sinh(2k)}\left\{\begin{array}{cc}|-\cosh(k(1-y'))\sinh(k(1+y))|,&\quad y\leq y';\\
\sinh(k(1-y))\cosh(k(1+y')),&\quad y\geq y'.\end{array}\right. \\
& \quad + \delta_0 \frac{|k| \abs{y-y'}^2}{\sinh(2k)}\left\{\begin{array}{cc}\cosh(k(1-y'))\cosh(k(1+y)),&\quad y\leq y';\\
\cosh(k(1-y))\cosh(k(1+y')),&\quad y\geq y'.\end{array}\right..
\end{align*}
Therefore, it follows that
\begin{align*}
\norm{rG_k}_{L^2_{y,y'}} & \lesssim \delta_0 \abs{k}^{-2}, \\
\norm{\partial_y(rG_k)}_{L^2_{y,y'}} + \norm{\partial_{y'} (rG_k)}_{L^2_{y,y'}} & \lesssim \delta_0 \abs{k}^{-1}, \\
\norm{\partial_{y} \partial_{y'} (rG)}_{L^2_{y,y'}} & \lesssim \delta_0. 
\end{align*}
Therefore, by Cauchy-Schwarz, we have
\begin{align*}
T_{13} \lesssim \delta_0\myr{ \abs{k}^2} \norm{\grad_k \phi_k}_{L^2}^2 \lesssim  \delta_0 \mathrm{D}_{k,\tau}, 
\end{align*}
which completes the desired estimate. 
\end{proof} 

Next, let us consider the $\bb T_{\tau \alpha}$, for which we prove the following.
\begin{lemma} \label{lem:LinTauAlpha}
Under the hypotheses of Proposition \ref{prop:lin}, there holds for $\delta_0$ sufficiently small,
\begin{align*}
\frac{d}{dt} \abs{k}^{-2/3}  \frac{d}{dt} \Re \langle (\nu^{1/3} \partial_y) \omega_k, \mathfrak{J}_k[(\nu^{1/3} \partial_y)\omega_k] \rangle + \frac{1}{12}\mathrm{D}_{k,\tau\alpha} \lesssim \mathrm{D}_{k,\alpha} + \left(\nu^{1/3} |k|^{2/3} \norm{\omega_k}_{L^2}^2 \right)^{1/2} \mathrm{D}_{k,\gamma}^{1/2}. 
\end{align*}
\siming{The implicit constant depends on $\|U\|_{C^3}.$}
\end{lemma}
\begin{proof}
Taking the derivative of \eqref{NSchannel}, we have
\begin{align*}
\left(\partial_t  + ikU - U'' ik\Delta_k^{-1} - \nu \Delta_k\right)\partial_y\omega_k & = -ikU' \omega_k + ik U''' \Delta_k^{-1}\omega_k+\siming{ikU''[\pa_y, \de_k^{-1}]\omega_k} \\
& =: \mathcal{C}_{k,1} + \mathcal{C}_{k,2}+\mathcal{C}_{k,3}. 
\end{align*}
It is important to note however, that the boundary conditions are different. In particular, $\partial_y \omega_k$ has \emph{Neumann} boundary conditions $\partial_{yy} \omega_k|_{y = \pm 1} = 0$.  
However, one can check that the contributions from the left-hand side behave similarly to the case of $\bb T_\tau$, and therefore we obtain 
\begin{align*} \abs{k}^{-2/3}&  \frac{d}{dt} \Re \langle (\nu^{1/3} \partial_y) \omega_k, \mathfrak{J}_k[(\nu^{1/3} \partial_y)\omega_k] \rangle + \frac{1}{8}\mathrm{D}_{k,\tau\alpha}  \lesssim \mathrm{D}_{k,\alpha}   + \sum_{j=1}^3\abs{k}^{-2/3} \nu^{2/3}\Re \brak{ \partial_y \omega_k, \mathfrak{J}_k[\mathcal{C}_{k,j}]}.  
\end{align*}
With straightforward estimates we have 
\begin{align*}
\norm{\mathcal{C}_{k,1}}_{L^2} + \norm{\mathcal{C}_{k,2}}_{L^2} \lesssim \abs{k} \norm{\omega_k}_{L^2}.
\end{align*}
\siming{To estimate the commutator $\mathcal{C}_{k,3}$, we explicitly write down the expression,
 \begin{align*}\norm{\mathcal{C}_{k,3}}_{L^2}\leq& |k|\|U''\|_{L^\infty}\lf\|\int_{-1}^1 \lf(\pa_y G_k(y,y')\omega_k(y')- G_k(y,y')\pa_{y'}\omega_k(y')\rg)dy'\rg\|_{L_y^2}\\
 \leq& |k|\|U''\|_{L^\infty}\lf\|\int_{-1}^1 \lf(\pa_y G_k(y,y')+ \pa_{y'}G_k(y,y')\rg)\omega_k(y')dy'\rg\|_{L_y^2}\\
 \lesssim &\|U''\|_{L^\infty}\||k|\omega_k\|_{L^2}\lesssim\delta_0\||k|\omega_k\|_{L^2}.
\end{align*} }
These estimates imply (using Lemma \ref{lem:BoundI}), 
\begin{align*}
\abs{k}^{-2/3}  \frac{d}{dt} \Re \langle (\nu^{1/3} \partial_y) \omega_k, \mathfrak{J}_k[(\nu^{1/3} \partial_y)\omega_k] \rangle + \frac{1}{8} \mathrm{D}_{k,\tau\alpha} & \lesssim \mathrm{D}_{k,\alpha} +  \nu^{2/3}\abs{k}^{1/3}\norm{\omega_k}_{L^2} \norm{\partial_y\omega_k}_{L^2} \\ 
& \lesssim \mathrm{D}_{k,\alpha} + \left(\nu^{1/3} |k|^{2/3} \norm{\omega_k}_{L^2}^2 \right)^{1/2} \mathrm{D}_{k,\gamma}^{1/2}, 
\end{align*}
which completes the lemma. 
\end{proof}

\subsection{Completing the Linearized Estimate} \label{sec:FinLin}
Putting together Lemmas \ref{lem:LinGamma}, \ref{lem:LinAlpha}, \ref{lem:LinBeta}, \ref{lem:LinTau}, and \ref{lem:LinTauAlpha} we obtain for some universal constant $K_0$, 
\begin{align*}
\frac{d}{dt}E_k + 2c_\beta \nu^{1/3} |k|^{2/3} \norm{\siming{\sqrt{U'}}\omega_k}_{L^2}^2 + \mathrm{D}_{k,\gamma} + c_\alpha\mathrm{D}_{k,\alpha} + \frac{c_\tau}{8} \mathrm{D}_{k,\tau} + \frac{c_\alpha c_\tau}{12} \mathrm{D}_{k,\tau \alpha} \\ 
& \hspace{-8cm} \leq K_0\delta_0 \mathrm{D}_{k,\tau} + c_\alpha K_0 \delta_0 \mathrm{D}_{k,\gamma} +  c_\alpha K_0\delta_0 \nu^{1/3} |k|^{2/3} \norm{\omega_k}_{L^2}^2 + K_0 c_\beta \mathrm{D}_{k,\gamma}^{1/2}\mathrm{D}_{k,\alpha}^{1/2} \\ & \hspace{-8cm} \quad + K_0 c_\tau \mathrm{D}_{k,\gamma} + K_0 c_\tau c_\alpha \mathrm{D}_{k,\alpha} + K_0 c_\tau c_\alpha \left(\nu^{1/3} |k|^{2/3} \norm{\omega_k}_{L^2}^2 \right)^{1/2} \mathrm{D}_{k,\gamma}^{1/2}. 
\end{align*}
Therefore, it follows that if 
\begin{align*}
c_\tau & < \frac{1}{32 K_0}, \quad
K_0 \delta_0  < \frac{c_\tau}{32}, \quad
c_\alpha  < \min\lf\{\frac{1}{8 K_0 \delta_0}\siming{,1}\rg\}, \quad 
\frac{c_\alpha}{c_\beta}  < \frac{1}{25 K_0}, \quad
\frac{c^2_\beta}{2c_\alpha}  < \frac{1}{25 K_0^2}. 
\end{align*}
then we have (assuming without loss of generality that $K_0 \geq 32$). 
\begin{align*}
K_0 c_\beta \mathrm{D}_{k,\gamma}^{1/2}\mathrm{D}_{k,\alpha}^{1/2} & \leq \frac{1}{10}\mathrm{D}_{k,\gamma} +  \frac{5}{2} K_0^2 c_\beta^2 \mathrm{D}_{k,\alpha} \leq \frac{1}{10}\mathrm{D}_{k,\gamma} +  \siming{\frac{1}{5}} c_\alpha \mathrm{D}_{k,\alpha}, \\
K_0 c_\tau c_\alpha \left(\nu^{1/3} |k|^{2/3} \norm{\omega_k}_{L^2}^2 \right)^{1/2} \mathrm{D}_{k,\gamma}^{1/2} & \siming{\leq \frac{c_\beta}{10} \nu^{1/3} |k|^{2/3} \norm{\omega_k}_{L^2}^2 + \frac{5}{2} \frac{c_\tau^2 c_\alpha^2K_0^2}{ c_\beta} \mathrm{D}_{k,\gamma} \leq \frac{c_\beta}{10} \nu^{1/3} |k|^{2/3} \norm{\omega_k}_{L^2}^2 + \frac{1}{10} \mathrm{D}_{k,\gamma}.} 
\end{align*}
then, we have the monotonicity estimate 
\begin{align*}
\frac{d}{dt} E_k + \siming{ \frac{1}{36}}\mathrm{D}_k +\siming{c(c_\tau,c_\al, c_\beta)}\nu^{1/3} |k|^{2/3} E_k \leq 0, 
\end{align*}
where $c(c_\tau,c_\al, c_\beta)>0$.
This in particular, also implies the stated exponential decay estimate. 
We also impose the condition
\begin{align*}
c_\beta^2 & \leq \frac{1}{4}c_\alpha + \frac{1}{4}(1 - c_\tau)
\end{align*}
to ensure the $E_k$ is coercive. 
While the procedure for setting the coefficients is by-now classical, we record an example for clarity (assuming without loss of generality that $K_0 > 32\siming{(1+\sup_{k\nq 0}\|\mathfrak{J}_k\|_{L^2\rightarrow L^2})}$), 
\begin{align*}
c_\tau  & = \frac{1}{64 K_0}, \qquad
\delta_0  < \frac{1}{(64 K_0)^2}, \qquad
c_\alpha  = K_0^{-9}, \qquad
c_\beta  = K_0^{-6}.
\end{align*}

%% file: nonlinear.tex
\subsection{Some Preliminary Lemmas} 

\begin{lemma} \label{lem:d2phi0}
We have the following estimates 
\begin{align*}
\siming{\norm{\partial_y\phi_0}_{L^\infty} \lesssim \mathcal{E}_0^{1/2}e^{-\delta\nu_\ast t}, }\quad \
\norm{\partial_y^2\phi_0}_{L^\infty} \lesssim \nu^{-1/6} \mathcal{E}_0^{1/2}e^{-\delta_\ast\nu t}. 
\end{align*}
Here $\mathcal{E}_0$ is defined in \eqref{mcl_E_0}.
\end{lemma}
\begin{proof}
By Gagliardo-Nirenberg-Sobolev inequality,  and elliptic regularity, we obtain
\begin{align*}
&\siming{
\norm{\partial_y\phi_0}_{L^\infty} \lesssim\lf\|\int_{-1}^1 \pa_y G_0(\cdot,y')\omega_0(t,y')dy'\rg\|_{L^\infty}\lesssim \norm{\omega_0}_{L^1}   \lesssim \mathcal{E}_0^{1/2}e^{-\delta_\ast\nu t},}\\\quad\
&\norm{\partial_y^2\phi_0}_{L^\infty} \lesssim \norm{\partial_y \omega_0}_{L^2}^{1/2}\norm{\omega_0}_{L^2}^{1/2} \lesssim \nu^{-1/6} \mathcal{E}_0^{1/2}e^{-\delta_\ast \nu t}. 
\end{align*}
Here $G_0(y,y')$ is the Green's kernel of $\de_y$ with Dirichlet boundary condition.
\siming{Since $G_0(y,y')$ is of the form $ a|y-y'|+by+cy'$ for some constants, we have the bound $\|\pa_y G_0\|_\infty\leq C. $}\end{proof} 

\begin{lemma} \label{lem:komegaD}
The following estimates hold
\begin{align*}
\sum_{k \neq 0} \abs{k}^{1+2m} \norm{\omega_k}_{L^2}^2 \lesssim& \nu^{-1/2} \mathcal{D}_\beta^{3/4} \mathcal{D}_\gamma^{1/4}e^{-2\dnt}, \\
\sum_{k\siming{\neq 0}} \abs{k}^{2m+4/3} \norm{\omega_k}_{L^2}^2 \lesssim& \nu^{-2/3} \mathcal{D}_{\gamma}^{1/2} \mathcal{D}_{\beta}^{1/2}e^{-2\dnt}. 
\end{align*}
\end{lemma}
\begin{proof}
Both estimates are proved using the same interpolation trick; let us just show the numerology for the first:
\begin{align*}
\sum_{k \neq 0} \nu^{1/2} \abs{k}^{1 + 2m} \norm{\omega_k}_{L^2}^2 &= \sum_{k \neq 0}  (\nu^{1/4}\abs{k}^{1/2 + \frac32 m})\norm{\omega_k}_{L^2}^{3/2} (\nu^{1/4}\abs{k}^{1/2+\frac12 m})\norm{\omega_k}_{L^2}^{1/2} \\
& \lesssim \left(\sum_{k \neq 0} \nu^{1/3} \abs{k}^{2/3+2m} \norm{\omega_k}_{L^2}^2 \right)^{3/4}\left(\sum_{k \neq 0} \nu \abs{k}^{2+2m} \norm{\omega_k}_{L^2}^2 \right)^{1/4} \\ 
& \lesssim \mathcal{D}_\beta^{3/4} \mathcal{D}_\gamma^{1/4}, 
\end{align*} 
which is the desired result. 

\end{proof}

\begin{lemma}\label{lem:dxphiL2t}
There hold the following estimates: 
\begin{align*}
\sum_{k\neq 0} |k|^{2m} \abs{k}^{3\mh{-\delta}} \norm{ \phi_k}_{L^\infty}^2 \lesssim \mathcal{D}_\tau e^{-2\dnt}. 
\end{align*}
Similarly,
\begin{align*}
\sum_{k\neq 0} |k|^{2m} \abs{k}^{3\mh{-\delta}} \norm{ (\nu^{1/3}\abs{k}^{-1/3}\partial_y) \phi_{k}}_{L^\infty}^2 & \lesssim \mathcal{D}_{\tau \alpha}e^{-2\dnt}. 
\end{align*}
Finally, for $m>2/3$, 
\begin{align}
\sum_{k\neq 0} \norm{\partial_y^2 \phi_{k}}_{L^\infty} & \lesssim \nu^{-1/6} \mathcal{E}^{1/2}_\nq e^{-\dnt}. \label{ineq:d2phiAlpha}
\end{align}
\end{lemma}
\begin{proof}
By Gagliardo-Nirenberg-Sobolev and Cauchy-Schwarz, 
\begin{align*}
\sum_{k\nq0} |k|^{2m} \abs{k}^{3\mh{-\delta}} \norm{\phi_k}_{L^\infty}^2 \lesssim \sum_{k\nq0} |k|^{2m} \abs{k}\mh{^{1-\delta}} \norm{|k|^{1/2} \partial_y \phi_k}_{L^2} \norm{|k|^{3/2} \phi_k}_{L^2} \lesssim \mathcal{D}_{\tau}e^{- 2\dnt}.
\end{align*}
A similar argument applies to the second estimate as well. 

Finally, to prove \eqref{ineq:d2phiAlpha}, we apply elliptic regularity
\begin{align*}\||k|^2\phi_k\|_{L^2}\lesssim \|\omega_k\|_{L^2},\quad 
\|\pa_y |k|^2\phi_k\|_{L^2}\lesssim \|\pa_y \omega_k\|_{L^2},
\end{align*}
which are natural consequences of testing the elliptic equation in \eqref{NSchannel} by functions $\{\phi_k, \pa_y\phi_k\}$. Combining them with the Gagliardo-Nirenberg interpolation yields that 
\begin{align*}
\sum_{k\neq0} &\norm{\partial_y^2 \phi_{k}}_{L^\infty} \leq\sum_{k\neq0}(\|\omega_k\|_{L^\infty} +\norm{|k|^2 \phi_{k}}_{L^\infty})  \lesssim \sum_{k\neq0} ( \norm{\partial_y \omega_{k}}_{L^2}^{1/2}\norm{\omega_{k}}_{L^2}^{1/2} +\norm{|k|^2\partial_y \phi_{k}}_{L^2}^{1/2}\||k|^2\phi_k\|_{L^2}^{1/2})\\
 \lesssim& \nu^{-\frac{1}{6}}\sum_{k\neq 0}|k|^{-m+\frac{1}{6}}\lf(\nu^{\frac{1}{6}}\||k|^{m-\frac{1}{3}}\pa_y\omega_k\|_{L^2}^{1/2}\rg)\||k|^m\omega_k\|_{L^2}^{1/2}\\
 \lesssim& \nu^{-\frac{1}{6}} \bigg(\sum_{k\nq 0} \abs{k}^{-2m +\frac{1}{3}}\bigg)^{1/2}\bigg(\sum_{k\nq 0}\norm{\nu^{\frac{1}{3}} \abs{k}^{m-\frac{1}{3}}\pa_y \omega_{k}}_{L^2} \norm{|k|^m\omega_{k}}_{L^2}\bigg)^{1/2}\lesssim \nu^{-\frac{1}{6}}\mathcal{E}^{1/2}e^{-\dnt} . 
\end{align*}
Here, we have used the constraint $m > 2/3$ to guarantee that $|k|^{-2m+\frac{1}{3}}\in \ell^1_{k\nq 0}$. Hence, we have the result. 
\end{proof}

\subsection{Zero Frequency Energy}
Expanding the linear and nonlinear contributions of the zero frequency energy gives
\begin{align}\label{ddtE0}
\myr{\frac{1}{2}}\frac{d}{dt} \mathcal{E}_0 = \mathcal{L}_0 + \mathcal{NL}_0\siming{+\delta_\ast \nu \mathcal{E}_0}, 
\end{align}
where
\begin{align*}
\mathcal{L}_0 & := e^{2\delta_\ast \nu t}\Re \brak{ \omega_{0},  \nu \partial_{yy}\omega_0}  + c_{\alpha}  e^{2\delta_\ast \nu t}\nu^{2/3} \Re\brak{\p_y \omega_{0}, \nu \partial_{yyy}\omega_0} \\ 
\mathcal{NL}_0 & :=  -e^{2\delta_\ast \nu t}\Re \brak{  \omega_{0},  (u\cdot \grad \omega)_0 } \myr{-} c_{\alpha} e^{2\delta_\ast \nu t} \nu^{2/3}\Re \brak{ \p_y \omega_{0},  \p_y  (u \cdot \grad \omega)_0}. 
\end{align*}
The estimates of these contributions are summarized in the following lemma. 
\begin{lemma} \label{lem:NLzero}
There exists a universal constant $C$ (independent of $\nu$) such that 
\begin{align}\label{est:NLzero}
\mathcal{L}_0 + \mathcal{NL}_0 \siming{+\delta_\ast \nu \mathcal{E}_0}&\leq -\frac{1}{4}\mathcal{D}_0 - 3 \delta_\ast\nu  \mathcal{E}_0 + \frac{C\mathcal{E}^{1/2}}{\sqrt{\nu}} \mathcal{D}.
\end{align}
\end{lemma}
\begin{proof}
First consider $\mathcal{L}_0$. 
Using that $\siming{\omega_0}(t,\pm 1) = \partial_{yy} \omega_0(t,\pm 1) = 0$, we may integrate by parts and obtain the following
\begin{align*}
\mathcal{L}_0 & = -e^{2\delta_\ast \nu t} \nu \norm{\partial_y \omega_{0}}_{L^2}^2  - c_{\alpha}  \nu^{5/3}e^{2\delta_\ast \nu t } \norm{\partial_y^2 \omega_{0}}_{L^2}^2. 
\end{align*}
By Poincar\'e inequality, for $\delta_\ast < 1/8$ we have
\begin{align}\label{NLzero_1}
\mathcal{L}_0 & \leq -\frac{1}{2}\mathcal{D}_0 - 4\siming{\delta_\ast} \nu \mathcal{E}_0. 
\end{align}
This completes the treatment of the linear contributions. 

Consider the nonlinear terms next. 
First, by the Biot-Savart law we have, 
\begin{align}
\mathcal{NL}_0 & = -e^{2\delta_\ast \nu t} \mathrm{Re} \brak{\omega_{0},\p_y\lf( \sum_{k \neq 0} \siming{ik} \phi_k \omega_{-k}\rg) }  + c_{\alpha} \nu^{2/3}e^{2\delta_\ast \nu t} \mathrm{Re} \brak{\p_y  \omega_{0}, \p_y^2\lf (\sum_{k \neq 0} \siming{ik} \phi_k \omega_{-k}\rg)}  =: T_0 + T_\alpha.\label{NLzero_2} 
\end{align}
Consider $T_0$ first.
Integrating by parts in $y$, (noting that $\omega_0$ vanishes on the boundary),  
\begin{align}
 | T_0| & = e^{2\delta_\ast \nu t}\lf| \mathrm{Re} \brak{\partial_y \omega_{0}, \sum_{k \neq 0} \siming{ik} \phi_k \omega_{-k} } \rg|\lesssim e^{2\delta_\ast \nu t} \norm{\partial_y \omega_{0}}_{L^2} \sum_{k \neq 0} \norm{\siming{ik} \phi_{k}}_{L^\infty} \norm{\omega_{-k}}_{L^2}
 \lesssim \frac{\mathcal{E}_{\neq}^{1/2}}{\sqrt{\nu}}\mathcal{D}^{1/2}_0 \mathcal{D}_\tau^{1/2},\label{NLzero_3}
\end{align}
which suffices by the bootstrap hypotheses. 

For $T_\alpha$ we have by integration by parts (here we are using both that $\siming{ik} \phi_k$ and $\omega_{-k}$ vanish on the boundary), 
\begin{align*}
|  T_\alpha |& \lesssim  e^{2\delta_\ast \nu t} \nu^{2/3} \abs{ \brak{\p_y \omega_{0},\p_y^2 \lf(\sum_{k \neq 0} \siming{ik}\phi_k \omega_{-k}\rg)} }   \lesssim e^{2\delta_\ast \nu t} \nu^{2/3} \abs{ \brak{\p_y^2 \omega_{0}, \p_y \lf(\sum_{k \neq 0}\siming{ik} \phi_k \omega_{-k}\rg)} } \\
& \lesssim e^{2\delta_\ast \nu t} \nu^{2/3} \norm{ \p_y^2 \omega_{0}}_{L^2} \left( \sum_{k \neq 0} \norm{\siming{ik}\partial_{y} \phi_k}_{L^\infty} \norm{\omega_{-k}}_{L^2} + \sum_{k \neq 0} \norm{\siming{ik} \phi_k}_{L^\infty} \norm{\partial_y \omega_{-k}}_{L^2}\right)=:T_{\al;1}+T_{\al;2}. 
\end{align*}
The first term is estimated using the following 
\begin{align*}
\sum_{k \neq 0} \norm{ik\partial_{y} \phi_k}_{L^\infty} \norm{\omega_{-k}}_{L^2} & \lesssim \sum_{k \neq 0} |k| \norm{\partial_{y}^2 \phi_k}_{L^2}^{1/2}\norm{\partial_{y} \phi_k}_{L^2}^{1/2} \norm{\omega_{-k}}_{L^2} \\ 
& \lesssim \sum_{k \neq 0} |k|^{1/3} \norm{|k|^m \omega_{-k}}_{L^2}^2  
 \lesssim \norm{|\pa_x|^{m+1/3} \omega_{\neq}}_{L^2}\norm{|\pa_x|^m\omega_{\neq}}_{L^2}, 
\end{align*}
provided $m > 1/3$, yielding the estimate
\begin{align*}
\siming{T_{\al;1}} & \lesssim \myr{\nu^{-1/3}}\mathcal{E}^{1/2} \mathcal{D}_0^{1/2} \mathcal{D}_{\beta}^{1/2}. 
\end{align*}
The $T_{\alpha;2}$-term is estimated with the following\siming{
\begin{align*}
 \sum_{k \neq 0}& \norm{ik \phi_k}_{L^\infty} \norm{\partial_y \omega_{-k}}_{L^2}  \lesssim \sum_{k \neq 0} \abs{k}^{-2m} \norm{|k|^{m+2} \phi_k}_{L^2}^{1/2} \norm{|k|^{m+2/3}\partial_y \phi_k}_{L^2}^{1/2} \norm{|k|^{m-1/3} \partial_y \omega_{-k}}_{L^2} \\
& \lesssim \nu^{-1/3}\mathcal{D}_\tau^{1/2} \mathcal{E}^{1/2}_\nq e^{-2\dnt}.
\end{align*}} Hence we have that 
\begin{align}\label{NLzero_4}
 | T_\alpha|\lesssim T_{\alpha;1}+T_{\alpha;2} \lesssim \myr{\nu^{-1/2}}\mathcal{E}^{1/2} \mathcal{D}.
\end{align}
yielding the desired estimate. Combining \eqref{NLzero_1}, \eqref{NLzero_2}, \eqref{NLzero_3}, \eqref{NLzero_4} yields the result \eqref{est:NLzero}.  
\end{proof} 

\subsection{Non-zero Frequency Energy}
First of all, we identify the expressions of the main terms to estimate. Define the following 
\begin{align}
\mathbb{L}_k:= &- U ik\omega_k +\myr{ U'' }ik \phi_k   + \nu \Delta_k \omega_k;\\
\mathbb{NL}_k:=& (\na^\perp\phi\cdot \na \omega)_k=\sum_{k'=-\infty}^\infty \na_{k-k'}^\perp \phi_k\cdot \na_{k'}\omega_{k'}.
\end{align}
Expanding the linear and nonlinear contributions of the nonzero frequency energy gives
\begin{align}\label{ddtEnq}
\frac{d}{dt} \mathcal{E}_\nq = \mathcal{L}_\nq + \mathcal{NL}_\nq\siming{+2\delta_\ast\nu^{1/3}\mathcal{E}_\nq}. 
\end{align}
Thanks to the definitions \eqref{def:Ek}, \eqref{mcl_E_0}, \eqref{mcl_E_nq}, \eqref{mcl_E},  \eqref{mcl_D_0},  \eqref{mcl_D_nq},  \eqref{mcl_D}, we obtain
\begin{align}\n
\mathcal{L}_\nq:= & 2e^{ 2\delta_\ast \nu^{1/3}t }\sum_{k\neq 0} |k|^{2m}\Re \brak{\omega_k, (I + c_\tau \mathfrak{J}_k)\mathbb{L}_k}\\
\n & +2c_\al e^{ 2\delta_\ast \nu^{1/3}t } \sum_{k\neq 0} |k|^{2m} \Re \brak{ \nu^{1/3} \abs{k}^{-1/3} \p_y \omega_{k}, (I + c_\tau \mathfrak{J}_k) \nu^{1/3} \abs{k}^{-1/3} \p_y \mathbb{L}_k} \\
&-c_\beta  e^{ 2\delta_\ast \nu^{1/3}t }\sum_{k\nq0} \nu^{1/3} \abs{k}^{2m-4/3} \bigg( \Re \brak{ik \omega_k,  \p_y \mathbb{L}_k}  + \Re \brak{ik  \mathbb{L}_k , \p_y \omega_k} \bigg);\label{L}\\
\n\mathcal{NL}_\nq:=&2e^{2\delta_\ast \nu^{1/3}t }\sum_{k\neq 0} |k|^{2m}\Re \brak{\omega_k, (I + c_\tau \mathfrak{J}_k)\mathbb{NL}_k}\\
\n &+2c_\al e^{2\delta_\ast \nu^{1/3}t } \sum_{k\neq 0} |k|^{2m} \Re \brak{ \nu^{1/3} \abs{k}^{-1/3} \p_y \omega_{k}, (I + c_\tau \mathfrak{J}_k) \nu^{1/3} \abs{k}^{-1/3} \p_y \mathbb{NL}_k} \\
\n&-c_\beta e^{2\delta_\ast\nu^{1/3}t} \sum_{k\nq0} \nu^{1/3} \abs{k}^{2m-4/3} \left( \Re \brak{ik \omega_k,  \p_y \mathbb{NL}_k} + \Re \brak{ik \mathbb{NL}_k , \p_y \omega_k} \right)\\
=:&\mathcal{NL}_{\gamma,\tau}+\mathcal{NL}_{\al,\tau\al}+\mathcal{NL}_{\beta}.\label{NL}
\end{align}
Here we group the $\gamma,\, \tau$ terms and $\al, \,\tau\al$ terms together because they share main features. 
The goal of this subsection is devoted to the proof of the following lemma.
\begin{theorem} \label{thm:NLNzero}
There exists a universal constant $C$ (independent of $\nu$) such that 
\begin{align*}
\mathcal{L}_{\neq} + \mathcal{NL}_{\neq}\siming{+2\delta_\ast\nu^{1/3}\mathcal{E}_\nq} &\leq -\siming{4\delta_\ast\mathcal{D}_{\neq} -6\delta_\ast \nu^{1/3}\sum_{k \neq 0}  e^{2\delta_\ast \nu^{1/3} t} \abs{k}^{2m + 2/3} E_k[\omega_k]+} \frac{C\mathcal{E}^{1/2}}{\sqrt{\nu}} \mathcal{D}.
\end{align*}
Here the $\mathcal{D}_\nq,\,\mathcal{D}$ are defined in \eqref{mcl_D_nq}, \eqref{mcl_D}.
\end{theorem}
We divide the proof of Theorem \ref{thm:NLNzero} into three lemmas. 
\begin{lemma}[Estimate of the $\mathcal{NL}_{\gamma,\tau}$]\label{lem:gatau} There exists a universal constant $C$ such that 
\begin{align}\label{est:gatau}
|\mathcal{NL}_{\gamma,\tau}|\leq \frac{C\mathcal{E}^{1/2}}{\sqrt{\nu}}\mathcal{D}.
\end{align}
\end{lemma}

\begin{lemma}[Estimate of the $\mathcal{NL}_{\al,\tau\al}$]\label{lem:altau} There exists a universal constant $C$ such that 
\begin{align}\label{est:altau}
|\mathcal{NL}_{\al,\tau\al}|\leq \frac{C\mathcal{E}^{1/2}}{\sqrt{\nu}}\mathcal{D}.
\end{align}
\end{lemma}

\begin{lemma}[Estimate of the $\mathcal{NL}_{\beta}$]\label{lem:beta} There exists a universal constant $C$ such that 
\begin{align}\label{est:beta}
|\mathcal{NL}_{\gamma,\tau}|\leq\frac{C\mathcal{E}^{1/2}}{\sqrt{\nu}}\mathcal{D}.
\end{align}
\end{lemma}
These lemmas implies Theorem \ref{thm:NLNzero}:
\begin{proof}[Proof of Theorem \ref{thm:NLNzero}] 
First of all, we observe that by Proposition \ref{prop:lin}, the linear part $\mathcal{L}_{\neq}$ has the bound \siming{Check!?}
\begin{align}
\mathcal{L}_\nq\leq -\siming{8\delta_\ast\mathcal{D}_\nq-8\delta_\ast \nu^{1/3} \sum_{k \neq 0}  e^{2\delta_\ast \nu^{1/3} t} \abs{k}^{2m + 2/3} E_k[\omega_k]}.
\end{align}
Moreover, the estimates presented in Lemma \ref{lem:gatau}, \ref{lem:altau}, \ref{lem:beta} yield that 
\begin{align*}
|\mathcal{NL}_\nq|\leq \frac{C\mathcal{E}^{1/2}}{\sqrt{\nu}}\mathcal{D}.
\end{align*}
So combining the two estimates above yields the conclusion. 
\end{proof}
The proof of each lemma are presented in the next three subsections. We further highlight that the notation $T_{\nq0},T_{0\nq}, T_{\nq\nq}, T_x,T_y,$ etc. will be redefined after the conclusion of each subsection. 

\subsubsection{The $\gamma$ and $\tau$ Contributions}
This subsection is devoted to the proof of Lemma \ref{lem:gatau}. 

We recall the definition of $\mathcal{NL}_{\gamma,\tau}$ \eqref{NL}:  
\begin{align}\label{gatau:NL}
\mathcal{NL}_{\gamma,\tau}=2e^{2\delta_\ast \nu^{1/3}t }\sum_{k\neq 0} |k|^{2m}\Re \brak{\omega_k, (I + c_\tau \mathfrak{J}_k)(\grad^\perp \phi \cdot \grad \omega)_k}=2e^{2\delta_\ast \nu^{1/3}t }(T_{0\nq}+T_{\nq 0}+T_{\nq\nq}),
\end{align}
where the terms $ T_{0\nq},\ T_{\nq 0},\ T_{\nq\nq}$ are defined as follows 
\begin{align*}
  T_{0 \neq}   & :=   \sum_{k \neq 0} |k|^{2m}\Re \brak{ \omega_{k}, (I + c_\tau \mathfrak{J}_k) (\partial_y\phi_0 \siming{ik}\omega_k)}, \\
\ T_{\neq0}   & :=   \sum_{k \neq 0} |k|^{2m} \Re \brak{ \omega_{k}, (I + c_\tau \mathfrak{J}_k) (\siming{ik}\phi_k \partial_y \omega_0)}, \\
  T_{\neq\neq} & := \sum_{\siming{\substack{k,k' \neq 0\\
  k\nq k'}}} |k|^{2m} \Re \brak{ \omega_{k}, (I + c_\tau \mathfrak{J}_k) (\grad_{k-k'}^\perp \phi_{k-k'} \cdot \grad_{k'} \omega_{k'})}. 
\end{align*}
A central challenge is dealing with the nonlinear contributions involving $\mathfrak{J}_k$. 

\noindent
\textbf{Treatment of the $T_{0\neq}$ term in \eqref{gatau:NL}:} \\ 
Through integration by parts, we observe that the first term in the $T_{0\nq}$ is zero. Hence,
\begin{align*}
T_{0 \neq} = \sum_{k \neq 0} |k|^{2m}\Re \brak{ \omega_{k}, c_\tau \mathfrak{J}_k (\partial_y\phi_0 ik \omega_k)}, 
\end{align*}
which is then estimated via H\"older's inequality and Lemma \ref{lem:BoundI}, yielding
\begin{align*}
 | T_{0 \neq}| \lesssim \sum_{k \neq 0} |k|^{2m + 1 \mh{-\delta}} \norm{\omega_{k}}^2_{L^2}  \norm{\partial_y\phi_0}_{L^\infty}. 
\end{align*}
Next 
by Lemma \ref{lem:d2phi0} and  \ref{lem:komegaD}, we have 
\begin{align*}
|T_{0\neq} |\lesssim \nu^{-1/2}\mathcal{E}_0^{1/2} \mathcal{D}_\beta^{3/4} \mathcal{D}_\gamma^{1/4}e^{- 2\dnt-\delta_\ast\nu t}, 
\end{align*}
which is consistent with  \eqref{est:gatau}.

\noindent
\textbf{Treatment of the $T_{\neq 0}$ term in \eqref{gatau:NL}:} \\  
We estimate this term with H\"older's inequality and Lemma \ref{lem:BoundI} to obtain 
\begin{align*}
|T_{\neq0} |& \lesssim \sum_{k\neq 0}\norm{  |k|^{ m} \omega_{k}}_{L^2} \norm{|k|^{m+1}\phi_k}_{L^\infty} \norm{\partial_y \omega_0}_{L^2}, 
\end{align*}
and hence by Lemma \ref{lem:dxphiL2t}, we have
\begin{align*}
|T_{\neq 0}| & \lesssim \nu^{-1/2}\mathcal{E}^{1/2} \mathcal{D}^{1/2}_\tau \mathcal{D}_0^{1/2}e^{- 2\dnt-\delta_\ast\nu t},
\end{align*}
and so this term is consistent with Theorem \ref{thm:NLNzero}.

\noindent
\textbf{Treatment of the $T_{\neq\neq}$ term in \eqref{gatau:NL}:} \\  
We rewrite the term with Lemma \ref{lem:AS} and further divide  as follows 
\begin{align}\n
   T_{\neq\neq} 
  & = \siming{ \sum_{\substack{k,k'\nq0\\ k\nq k'}}  \Re \brak{ (1+c_\tau \mathfrak{J}_k )\abs{k}^m  \omega_{k}, \abs{k}^m(\grad_{k-k'}^\perp \phi_{k-k'}\cdot \grad_{k'}\omega_{k'})}}
 \\ \n
& = \sum_{\siming{\substack{k,k' \neq 0\\
  k\nq k'}}} |k|^{2m} \Re \brak{(1+ c_\tau \mathfrak{J}_k) \omega_{k},  \partial_y( i(k-k') \phi_{k-k'} \omega_{k'})}   - \sum_{\siming{\substack{k,k' \neq 0\\
  k\nq k'}}} |k|^{2m} \Re \brak{ (1+ c_\tau \mathfrak{J}_k )\omega_{k}, ik( \partial_y \phi_{k-k'} \omega_{k'})} \\
  & =: T_{ x} + T_{ y}.\label{gatau_Tnqnq} 
\end{align}
To estimate $T_{x}$ we use integration by parts, Cauchy-Schwarz, and Lemma \ref{lem:BoundI}, \ref{lem:CommIpy}, \ref{lem:H} to obtain the following, using that $m \in (1/2,1)$, 
\begin{align*}
T_{ x} & \siming{\lesssim \sum_{\siming{\substack{k,k' \neq 0\\
  k\nq k'}}} |k|^{2m} (1+\|\mathfrak{J}_k\|_{L^2\rightarrow L^2}+|k|^{-1}\|[\mathfrak{J}_k,\pa_y]\|_{L^2\rightarrow L^2})\norm{\na_k \omega_{k}}_{L^2}  \norm{(k-k')\phi_{k-k'}}_{L^\infty} \norm{\omega_{k'}}_{L^2} }\\
    & \lesssim \sum_{\siming{\substack{k,k' \neq 0\\
  k\nq k'}}} \abs{k}^{m} \norm{\grad_k \omega_{k}}_{L^2} \left(\abs{k-k'}^m + \abs{k'}^m \right)  \norm{(k-k')\phi_{k-k'}}_{L^\infty} \norm{\omega_{k'}}_{L^2} \\
& \lesssim \nu^{-1/2}\mathcal{D}_\gamma^{1/2} \left( \sum_{k \neq 0} \norm{\abs{k-k'}^{1 + m }\phi_{k-k'}}_{L^\infty}^2 \right)^{1/2} \left( \sum_{k \neq 0}\norm{\omega_{k}}_{L^2} \right) e^{-\dnt}\\
& \quad + \nu^{-1/2}\mathcal{D}_\gamma^{1/2} \left( \sum_{k \neq 0} \norm{\abs{k-k'} \phi_{k-k'}}_{L^\infty} \right) \left( \sum_{k \neq 0}\abs{k}^{m}\norm{\omega_{k}}_{L^2}^2 \right) e^{-\dnt}\\
& \lesssim \frac{\mathcal{E}^{1/2}}{\sqrt{\nu}} \mathcal{D}_\tau^{1/2} \mathcal{D}_{\gamma}^{1/2}e^{-3\dnt},
\end{align*}
where in the last line we used $m > 1/2$ to convert the $\ell^1$ sums into $\ell^2$ and then used Lemma \ref{lem:dxphiL2t}.
This is hence consistent with  \eqref{est:gatau} and Theorem \ref{thm:NLNzero}. 

To estimate $T_{ y}$, we again use Cauchy-Schwarz, and Lemma \ref{lem:BoundI} to obtain the following and introduce a frequency decomposition 
\begin{align*}
T_{y} & \lesssim \sum_{ \substack{k,k' \neq 0\\
  k\nq k'}} \left(\mathbf{1}_{\abs{k-k'} < \abs{k'}} + \mathbf{1}_{\abs{k-k'} \geq \abs{k'}} \right) \norm{\abs{k}^{m+1/3} \omega_{k}}_{L^2}  \abs{k}^{m + 2/3 \mh{- \delta}} \norm{\partial_y \phi_{k-k'}}_{L^\infty} \norm{\omega_{k'}}_{L^2} \\
& =: T_{y;LH} + T_{ y;HL}.
\end{align*}
Estimating the $T_{y;LH}$ terms we have by Lemma \ref{lem:komegaD} 
\begin{align*}
T_{y;LH} & \lesssim  \sum_{ \substack{k,k' \neq 0\\
  k\nq k'}}  \mathbf{1}_{\abs{k-k'} < \abs{k'}}  \norm{\abs{k}^{m+1/3} \omega_{k}}_{L^2}  \abs{k'}^{m + 2/3\mh{ - \delta}} \norm{\partial_y \phi_{k-k'}}_{L^\infty} \norm{\omega_{k'}}_{L^2} \\
& \lesssim \nu^{-1/3-1/6} \mathcal{E}^{1/2}  \mathcal{D}_\beta^{3/4} \mathcal{D}_\gamma^{1/4}e^{-3\dnt}, 
\end{align*}
where in the last line we used that by the Gagliardo-Nirenberg-Sobolev inequality and $m > 1/2$  we have, 
\begin{align*}
\sum_{k \neq 0} \norm{ \partial_y \phi_{k}}_{L^\infty} & \lesssim \sum_{k\nq 0} \norm{ \partial_y^2 \phi_{k}}_{L^2}^{1/2}\norm{\partial_y \phi_{k}}_{L^2}^{1/2} \lesssim \mathcal{E}^{1/2}e^{-\dnt}. 
\end{align*}
Therefore, $T_{y;LH}$ is consistent with Theorem \ref{thm:NLNzero}.
For $T_{y;HL}$ we instead make the estimate
\begin{align*}
T_{y;HL}  & \lesssim \sum_{ \substack{k,k' \neq 0\\
  k\nq k'}} \mathbf{1}_{\abs{k-k'} \geq \abs{k'}} \norm{\abs{k}^{m+1/3} \omega_{k}}_{L^2}  \abs{k-k'}^{m + 2/3\mh{ - \delta}} \norm{\partial_y \phi_{k-k'}}_{L^\infty} \norm{\omega_{k'}}_{L^2}. 
\end{align*} 
Next, we note that by the Gagliardo-Nirenberg-Sobolev inequality,  H\"older's inequality, elliptic regularity, and Lemma \ref{lem:dxphiL2t} we have, 
\begin{align*} 
\sum_{k\neq0} \abs{k}^{m+2/3 \mh{-\delta}} \norm{ \partial_y \phi_{k}}_{L^\infty} & \lesssim \sum_{k\nq 0} \abs{k}^{-1/2\mh{-\delta}} \abs{k}^{m} \norm{ \abs{k}^{1/3} \partial_y^2 \phi_{k}}_{L^2}^{1/2}\norm{\abs{k}^{5/6} \abs{k}^{-1/3}\partial_y \phi_{k}}_{L^2}^{1/2} \\
& \lesssim \nu^{-1/12-1/6} \mathcal{D}_\beta^{1/4} \mathcal{D}_{\tau \alpha}^{1/4}e^{-\dnt}. 
\end{align*}
Therefore, 
\begin{align*}
T_{y;HL}  & \lesssim \nu^{-5/12} \mathcal{E}^{1/2} \mathcal{D}_\beta^{3/4} \mathcal{D}_{\tau \alpha}^{1/4}e^{-3\dnt}, 
\end{align*}
which is consistent with  \eqref{est:gatau} and Theorem \ref{thm:NLNzero}.  This completes the treatment of the $T_{\neq\neq}$ terms. 

\ifx
Next we consider the $T_{\nq\nq;\gamma }$ term in \eqref{gatau_Tnqnq}. We observe that thanks to the boundary condition $\phi_k\big|_{y=\pm 1}=\omega_k\big|_{y=\pm 1}=0$, there is a null structure (divergence-free condition),
\begin{align}
 \sum_{\substack{k,k'\nq0\\ k\nq k'}}  \Re \brak{ \abs{k}^m  \omega_{k},  \grad_{k-k'}^\perp \phi_{k-k'}\cdot \grad_{k'}(|k'|^m\omega_{k'}) } =0.
\end{align}Hence  the $T_{\neq\neq;\gamma}$ term can be rephrased and divided as follows
\begin{align*}
T_{\neq\neq;\gamma} & = \sum_{ \substack{k,k' \neq 0\\
  k\nq k'}}  \Re \brak{ \abs{k}^m \omega_{k}, \lf(\abs{k}^m - \abs{k'}^m\rg) \bigg(\partial_y \phi_{k-k'} ik'\omega_{k'} - i(k-k') \phi_{k-k'} \partial_y\omega_{k'}\bigg) }  
  =: T_{\gamma y} + T_{\gamma x}.
\end{align*}
Consider the $T_{\gamma y}$ term first, which is the more delicate term to estimate.
We first use a frequency decomposition
\begin{align*}
T_{\gamma y} & = \sum_{ \substack{k,k' \neq 0\\
  k\nq k'}} (\mathbf{1}_{\abs{k-k'} < \abs{k'}/2} + \mathbf{1}_{\abs{k-k'} \geq \abs{k'}/2} ) \Re \brak{ \abs{k}^m \omega_{k}, (\abs{k}^m - \abs{k'}^m) \partial_y \phi_{k-k'} ik'\omega_{k'} } \\
& =: T_{\gamma y;LH} + T_{\gamma y;HL}. 
\end{align*}
For the $T_{\gamma y;HL}$ term the commutator is not relevant. Note that on the support of the summand, we have $\abs{k'} \lesssim \abs{k-k'}$, and hence we have
\begin{align*}
T_{\gamma y;HL} \lesssim \sum_{ \substack{k,k' \neq 0\\
  k\nq k'}}  \mathbf{1}_{\abs{k-k'} \geq \abs{k'}/2} \norm{\abs{k}^{m} \omega_k}_{L^2} \norm{ \abs{k-k'}^{2/3+m}\partial_y \phi_{k-k'}}_{L^\infty} \norm{\abs{k'}^{1/3}\omega_{k'}}_{L^2}. 
\end{align*}
Note that by Lemma \ref{lem:dxphiL2t},
\begin{align*}
\left(\sum_{k\neq 0} |k|^{2m}\norm{ \abs{k}^{2/3}\partial_y \phi_{k}}_{L^\infty}^2\right)^{1/2} & \lesssim \left( \sum_{k\neq 0} \abs{k}^{2m} \norm{ k (\abs{k}^{-1/3} \partial_y) \phi_{k}}_{L^\infty}^2 \right)^{1/2}  
  \lesssim \nu^{-1/3}\mathcal{D}_{\tau \alpha}^{1/2}e^{-\dnt}. 
\end{align*}
Therefore, by Cauchy-Schwarz and $m > 1/2$, we have 
\begin{align*}
T_{\gamma y;HL} \lesssim \nu^{-1/2} \mathcal{E}^{1/2} \mathcal{D}_{\tau \alpha}^{1/2} \mathcal{D}_{\beta}^{1/2}e^{-3\dnt}, 
\end{align*}
which is consistent with \eqref{est:gatau} and Theorem \ref{thm:NLNzero}. 

For $T_{\gamma y;LH}$ we use the following estimate, which holds on the support of the summand
\begin{align*}
\abs{\abs{k}^m - \abs{k'}^m} \lesssim \abs{k-k'} \abs{k'}^{m-1}
\end{align*}
to obtain 
\begin{align*}
T_{\gamma y;LH} & \lesssim \sum_{ \substack{k,k' \neq 0\\
  k\nq k'}} \mathbf{1}_{\abs{k-k'} < \abs{k'}/2}  \norm{\abs{k}^{m} \omega_k}_{L^2} \norm{ \abs{k-k'}\partial_y \phi_{k-k'}}_{L^\infty} \norm{\abs{k'}^{m} \omega_{k'}}_{L^2} \\
& \lesssim \sum_{k,k' \neq 0} \mathbf{1}_{\abs{k-k'} < \abs{k'}/2}  \norm{\abs{k}^{m} \omega_k}_{L^2} \norm{ \abs{k-k'}^{2/3 }\partial_y \phi_{k-k'}}_{L^\infty} \norm{\abs{k'}^{m+1/3} \omega_{k'}}_{L^2}.  
\end{align*}
Then we note that by Lemma \ref{lem:dxphiL2t} and $m > 1/2$
\begin{align*}
  \sum_k \norm{ k (\abs{k}^{-1/3}\partial_y) \phi_k}_{L^\infty} & \lesssim \left(\sum_k |k|^{2m} \norm{ k (\abs{k}^{-1/3}\partial_y) \phi_k}_{L^\infty}^2\right)^{1/2} \\
  & \lesssim \nu^{-1/3}\mathcal{D}_{\tau\alpha}^{1/2}e^{-\dnt}. 
\end{align*}
Therefore, we have
\begin{align*}
T_{\gamma y;HL} & \lesssim \nu^{-1/2} \mathcal{E}^{1/2} \mathcal{D}_{\tau \alpha}^{1/2} \mathcal{D}_{\beta}^{1/2}e^{-3\dnt}, 
\end{align*}
which is consistent with \eqref{est:gatau} and Theorem \ref{thm:NLNzero}, completing the treatment of the $T_{\gamma y}$ term.

We turn next to the $T_{\gamma x}$ term, which is more straightforward.
Indeed, we do not need a commutator and instead obtain using H\"older's inequality followed by Lemma \ref{lem:dxphiL2t} (using that $m > 1/2$)
\begin{align*}
  T_{\gamma x} & \lesssim \sum_{ \substack{k,k' \neq 0\\
  k\nq k'}} \norm{\abs{k}^m \omega_k}_{L^2} \norm{\abs{k-k'} \phi_{k-k'}}_{L^\infty} \norm{ \abs{k'}^m\partial_y\omega_{k'} }_{L^2} \\
& \quad + \sum_{\substack{k,k' \neq 0\\
  k\nq k'}} \norm{\abs{k}^m \omega_k}_{L^2} \norm{\abs{k-k'}^{1+m} \phi_{k-k'}}_{L^\infty} \norm{\partial_y\omega_{k'} }_{L^2} \\
& \lesssim \nu^{-1/2}\mathcal{E}^{1/2} \mathcal{D}_\tau^{1/2} \mathcal{D}^{1/2}_\gamma e^{-3\dnt}, 
\end{align*}
which is consistent with  \eqref{est:gatau} and Theorem \ref{thm:NLNzero}.

\fi
\subsubsection{The $\alpha$ and $\alpha\tau$ Contributions}
As in the previous subsection we begin with the same decomposition of the term $\mathcal{NL}_{\al,\tau\al}$ \eqref{NL}: 
\begin{align}
\frac{1}{2c_\al}e^{-2\dnt}&\mathcal{NL}_{\al,\tau\al}=T_{0\nq}+T_{\nq0}+T_{\nq\nq} ,\label{altau:NL}\\
\n  T_{0 \neq}   & :=   \sum_{k\neq 0} |k|^{2m}\Re \brak{ \nu^{1/3} \abs{k}^{-1/3} \p_y \omega_{k}, (I + c_\tau \mathfrak{J}_k) \nu^{1/3} \abs{k}^{-1/3} \p_y (\partial_y\phi_0 ik \omega_k)} ,\\
\n T_{\neq0}   & :=   \sum_{k\neq 0} |k|^{2m} \Re \brak{ \nu^{1/3} \abs{k}^{-1/3} \p_y \omega_{k}, (I + c_\tau \mathfrak{J}_k) \nu^{1/3} \abs{k}^{-1/3} \p_y (ik\phi_k \partial_y \omega_0)} ,\\
 \n T_{\neq\neq} &: = \sum_{ \substack{k,k' \neq 0\\
  k\nq k'}} |k|^{2m} \Re \brak{ \nu^{1/3} \abs{k}^{-1/3} \p_y \omega_{k}, (I + c_\tau \mathfrak{J}_k) \nu^{1/3} \abs{k}^{-1/3} \p_y (\grad_{k-k'}^\perp \phi_{k-k'} \cdot \grad_k \omega_{k'})}. 
\end{align}
We note that all the $T_{(\ast)}$ are refreshed in this subsection.

\noindent
\textbf{Treatment of the $T_{0 \neq}$ term in \eqref{altau:NL}:} \\
Analogous to the treatment in the previous section we have by integration by parts (using that $\omega_k$ vanishes along the boundaries)
\begin{align}\n
T_{0 \neq}   & = c_\tau   \sum_{k\neq 0} |k|^{2m}\Re \brak{ \nu^{1/3} \abs{k}^{-1/3} \p_y \omega_{k}, \mathfrak{J}_k \nu^{1/3} \abs{k}^{-1/3} \p_y (\partial_y\phi_0 ik \omega_k)} \\ \n
& \quad +\sum_{k\neq 0} |k|^{2m}\Re \brak{ \nu^{1/3} \abs{k}^{-1/3} \p_y \omega_{k}, \nu^{1/3} \abs{k}^{-1/3} \partial_y^2\phi_0 ik \omega_k)} \\ \n
& \quad +\sum_{k\neq 0} |k|^{2m}\Re \brak{ \nu^{1/3} \abs{k}^{-1/3} \p_y \omega_{k}, \nu^{1/3} \abs{k}^{-1/3} \partial_y\phi_0 ik \partial_y \omega_k)}\\
&=:T_{0\neq;1}+T_{0\neq;2}+T_{0\neq;3}.\label{altau_0nq}
\end{align}
The last term vanishes since $\partial_y \phi_0$ is purely real
\begin{align*}
T_{0\nq;3}=\sum_{k\neq 0} |k|^{2m}\Re \ ik\int  |\nu^{1/3} \abs{k}^{-1/3} \p_y \omega_{k}|^2\partial_y\phi_0 dy= 0. 
\end{align*}

For the $T_{0\neq;2}$ term in \eqref{altau_0nq}, we have
\begin{align*}
T_{0\neq;2} \lesssim \nu^{2/3}\norm{\abs{k}^{m} \p_y \omega_k }_{L^2}\norm{\abs{k}^{m+1/3}\omega_k }_{L^2} \norm{\partial_y^2\phi_0}_{L^\infty}. 
\end{align*}
Since we have that $\pa_{yy}\phi_0=\omega_0$, $\|\pa_{yy}\phi_0\|_{L^\infty}=\|\omega_0\|_{L^\infty}$. By Gagliardo-Nirenberg-Sobolev, we have 
\begin{align*}
\norm{\partial_y^2\phi_0}_{L^\infty} & \lesssim \norm{\partial_y\omega_0}_{L^2}^{1/2}\norm{\omega_0}_{L^2}^{1/2}  \lesssim \nu^{-1/6} \mathcal{E}_0^{1/2}e^{-\delta_\ast\nu t}. 
\end{align*}
Therefore,
\begin{align*}
T_{0\neq;2} \lesssim \nu^{-1/6}\mathcal{E}_0^{1/2} \mathcal{D}_{\gamma}^{1/2} \mathcal{D}_\beta^{1/2}e^{-2\dnt-\delta_\ast\nu t}, 
\end{align*}
which is consistent with \eqref{est:altau} and Theorem \ref{thm:NLNzero}.

For the $T_{0\neq;1}$ term in \eqref{altau_0nq}, we first apply the symmetry of $\mathfrak{J}_k$ (Lemma \ref{lem:AS}) and integration by parts (using that $\omega_k$ vanishes on the boundary), to obtain that 
\ifx
\begin{align*}
T_{0\neq;1} & =  c_\tau\sum_{k\neq 0} |k|^{2m}\Re \brak{ \nu^{1/3} \abs{k}^{-1/3} \mathfrak{J}_k\p_y \omega_{k},   \nu^{1/3} \abs{k}^{-1/3} \p_y (\partial_y\phi_0 ik \omega_k)} \\
&\lesssim c_{\tau}\sum_{k\neq 0} \norm{\nu^{1/3} \abs{k}^{m-1/3} \p_y \omega_{k}}_{L^2} \nu^{1/3} \abs{k}^{m+2/3}\lf( \norm{\omega_k}_{L^2} \norm{\partial_y^2\phi_0}_{L^\infty}+\norm{\pa_y\omega_k}_{L^2} \norm{\partial_y\phi_0}_{L^\infty}\rg) \\
& \lesssim  c_\tau \frac{\mathcal{E}^{1/2}}{ {\nu^{1/6}}}\sum_{k \neq 0} \nu^{1/3} |k|^{2/3} \abs{k}^{2m} E_k+???\\
&\lesssim c_\tau \frac{\mathcal{E}^{1/2}}{\myr{\nu^{1/6}}}\mathcal{D}_E,  
\end{align*}\fi
\siming{ 
\begin{align*}
T_{0\neq;1} & =  c_\tau\sum_{k\neq 0} |k|^{2m}\Re \brak{ \nu^{1/3} \abs{k}^{-1/3} \mathfrak{J}_k\p_y \omega_{k},   \nu^{1/3} \abs{k}^{-1/3} \p_y (\partial_y\phi_0 ik \omega_k)} \\
& = - c_\tau\sum_{k\neq 0} |k|^{2m}\Re \brak{ \nu^{1/3} \abs{k}^{-1/3}\p_y \mathfrak{J}_k\p_y \omega_{k},   \nu^{1/3} \abs{k}^{-1/3}  (\partial_y\phi_0 ik \omega_k)} 
\end{align*}
Then we use Lemma \ref{lem:BoundI}, \ref{lem:CommIpy} to estimate the commutators,
\begin{align*}
T_{0\nq;1}&\lesssim c_{\tau}\sum_{k\neq 0} \nu^{1/3} \abs{k}^{m-1/3}\lf(\|\mathfrak{J}_k\|_{L^2\rightarrow L^2} +|k|^{-1}\|[\mathfrak{J}_k,\pa_y]\|_{L^2\rightarrow L^2}\rg)\norm{\na_k\p_{ y} \omega_{k}}_{L^2}\\
&\qquad\qquad\times \nu^{1/3} \abs{k}^{m+2/3}\norm{ \omega_k}_{L^2} \norm{\partial_y\phi_0}_{L^\infty}  \\
&\lesssim \frac
{c_{\tau}}{\nu^{1/2}}\sum_{k\neq 0} \lf(\nu\norm{\nu^{1/3} \abs{k}^{m-1/3}\na_k\p_{ y} \omega_{k}}_{L^2}^2\rg)^{1/2}\lf( \nu \norm{\abs{k}^{m+1} \omega_k}_{L^2}^2\rg)^{1/4}\\
&\qquad\qquad\times\lf(\nu^{1/3} |k|^{2m+2/3}\norm{\omega_k}_{L^2}^2 \rg)^{1/4}\norm{\partial_y\phi_0}_{L^\infty}.
\end{align*}
Now we recall Lemma \ref{lem:d2phi0} and the definition of $ \mathcal{D}_{\al},\,\mathcal{D}_{\beta},\, \mathcal{D}_{\gamma}$ \eqref{def:Dks}, \eqref{def:Dkscl}, and then obtain
\begin{align}
T_{0\nq;1}&\lesssim   \frac{\mathcal{E}^{1/2}}{\myr{\nu^{1/2}}}\mathcal{D}_\al^{1/2}\mathcal{D}_\beta^{1/4}\mathcal{D}_{\gamma}^{1/4}e^{-2\dnt-\delta_\ast \nu t}, \label{prototype} 
\end{align} 
}
which is consistent with Theorem \ref{thm:NLNzero}. This completes the treatment of the $T_{0\neq}$ term.

\noindent
\textbf{Treatment of the $T_{\neq 0}$ term in \eqref{altau:NL}:} \\
We use Lemma \ref{lem:dxphiL2t} to obtain
\begin{align*}
T_{\neq0}   & \lesssim   \sum_{k\neq 0} |k|^{2m}\norm{\nu^{1/3} \abs{k}^{-1/3} \p_y^2 \omega_{k}}_{L^2} \norm{\abs{k}^{2/3+m} \phi_k}_{L^\infty} \norm{\nu^{1/3}\partial_y \omega_0}_{L^2}\lesssim \nu^{-1/2} \mathcal{E}_0^{1/2} \mathcal{D}^{1/2}_\gamma\mathcal{D}_\tau^{1/2}, 
\end{align*}
which is consistent with Theorem \ref{thm:NLNzero}. This completes the treatment of the $T_{\neq0}$ term.

\noindent
\textbf{Treatment of the $T_{\neq\neq}$ term in \eqref{altau:NL}:} \\
The most troublesome term is the $T_{\neq\neq}$ term, which we treat now.
We apply the symmetric property of $\mathfrak{J}_k$ (Lemma \ref{lem:AS}) to rewrite the term as follows, 
\begin{align*}
T_{\neq\neq} & = \sum_{ \substack{k,k' \neq 0\\
  k\nq k'}} \Re \brak{ \nu^{1/3} \abs{k}^{m-1/3} (1+c_\tau\mathfrak{J}_k) \p_y \omega_{k},   \nu^{1/3} \abs{k}^{m-1/3} \p_y (\grad_{k-k'}^\perp \phi_{k-k'} \cdot \grad_{k'} \omega_{k'})} .
\end{align*}\ifx
  \\
& \quad + \sum_{ \substack{k,k' \neq 0\\
  k\nq k'}} \Re \brak{ \nu^{1/3} \abs{k}^{m-1/3} \p_y \omega_{k},\nu^{1/3} \myr{(\abs{k}^{m-1/3}-\abs{k'}^{m-1/3})} \p_y (\grad_{k-k'}^\perp \phi_{k-k'} \cdot \grad_{k'} \omega_{k'})} \\
& \quad + \sum_{ \substack{k,k' \neq 0\\
  k\nq k'}} \Re \brak{ \nu^{1/3} \abs{k}^{m-1/3} \p_y \omega_{k},  \p_y \lf(\grad_{k-k'}^\perp \phi_{\myr{k-k'}} \cdot \grad_{k'} (\nu^{1/3} \abs{k'}^{m-1/3} \omega_{k'})\rg)} \\
& = T_{\neq\neq;\tau} + T_{\neq\neq;\alpha 1}+ T_{\neq\neq;\alpha 2 }. \siming{The last two terms are not estimated.}\fi
The estimate of $T_{\neq\neq}$ begins by integrating by parts in the term involving $\partial_y\phi_{k-k'} ik'\omega_{k'}$, and applying H\"older's inequality,
\begin{align*}
T_{\neq\neq} =&- \sum_{ \substack{k,k' \neq 0\\
  k\nq k'}} \Re \brak{ \nu^{1/3} \abs{k}^{m-1/3} (1+c_\tau\mathfrak{J}_k)\p_y \omega_{k},   \nu^{1/3} \abs{k}^{m-1/3} \p_y (i(k-k') \phi_{k-k'} \pa_y \omega_{k'})} \\
  & -\sum_{ \substack{k,k' \neq 0\\
  k\nq k'}} \Re \brak{ \nu^{1/3} \abs{k}^{m-1/3} \p_y(1+c_\tau\mathfrak{J}_k)\p_y \omega_{k},  \nu^{1/3} \abs{k}^{m-1/3}  (\pa_y \phi_{k-k'}  i{k'} \omega_{k'})} \\
 \lesssim\mh{_\delta}& \sum_{ \substack{k,k' \neq 0\\
  k\nq k'}} \norm{\nu^{1/3} |k|^{m-1/3} (1+c_\tau\mathfrak{J}_k)\p_y \omega_{k}}_{L^2}  \nu^{1/3} |k|^{m-1/3\mh{-\delta}} \\
   & \quad\quad\quad \times \left( \norm{(k-k') \partial_y \phi_{k-k'}}_{L^\infty} \norm{\partial_y \omega_{k'}}_{L^2} + \norm{(k-k')\phi_{k-k'}}_{L^\infty} \norm{\partial_y^2 \omega_{k'}}_{L^2}\right) \\
&  + \sum_{ \substack{k,k' \neq 0\\
  k\nq k'}}\siming{ \norm{\nu^{1/3} |k|^{m-1/3} \pa_y(1+c_\tau\mathfrak{J}_k)\p_y \omega_{k}}_{L^2}  \nu^{1/3} |k|^{m-1/3\mh{-\delta}}\norm{\partial_y \phi_{k-k'}}_{L^\infty} \norm{k' \omega_{k'}}_{L^2}} 
=:  \sum_{j=1}^3 T_{\nq\nq; j}. 
\end{align*}
\siming{
We note that to justify that $\mathfrak{J}_k(\pa_y\omega_k)$ is well-defined on the boundary, one can apply the Lemma \ref{lem:CommIpy} and \ref{lem:H} to get $\|\pa_y \mathfrak{J}_k f_k\|_{L^2}\leq \| \mathfrak{J}_k\pa_y f_k\|_{L^2}+\|[\pa_y, \mathfrak{J}_k] f_k\|_{L^2}\leq \|f_k\|_{H_k^1}<\infty$.}
To estimate the term $T_{\nq\nq;1},$ we first apply Lemma \ref{lem:BoundI}, and introduce a frequency decomposition
\begin{align*}
T_{\nq\nq; 1} & \lesssim \sum_{ \substack{k,k' \neq 0\\
  k\nq k'}} (\mathbf{1}_{\abs{k-k'} < \abs{k'}/2} + \mathbf{1}_{\abs{k-k'} \geq \abs{k'}/2}) \norm{\nu^{1/3} |k|^{m-1/3} \p_y \omega_{k}}_{L^2}  \nu^{1/3} |k|^{m-1/3\mh{-\delta}} \\
& \qquad \times \norm{(k-k') \partial_y \phi_{k-k'}}_{L^\infty} \norm{\partial_y \omega_{k'}}_{L^2} \\ 
& = : T_{1;LH} + T_{1;HL}. 
\end{align*}
For the $LH$ term we have 
\begin{align*}
T_{1;LH} 
& \lesssim \sum_{\substack{k,k' \neq 0\\
  k\nq k'}} \mathbf{1}_{\abs{k-k'} < \abs{k'}/2} \norm{\nu^{1/3} |k|^{m-1/3} \p_y \omega_{k}}_{L^2} \norm{\abs{k-k'}\mh{^{1-\delta}}\partial_y \phi_{k-k'}}_{L^\infty}\frac{1}{\nu^{1/6}}\norm{\nu^{1/2} \abs{k'}^{m-1/3} \partial_y \omega_{k'}}_{L^2}, 
\end{align*}
Note that for $m >1/2$ we have by Lemma \ref{lem:dxphiL2t}, 
\begin{align*}
\sum_{k\neq 0} \norm{\abs{k}\mh{^{1-\delta}} \partial_y \phi_{k}}_{L^\infty} & \lesssim \left(\sum_{k\neq 0} \abs{k}^{2m + 3 \mh{- 2\delta}} \norm{\abs{k}^{-1/3}\partial_y \phi_{k}}_{L^\infty}^2 \right)^{1/2}  \lesssim \nu^{-1/3} \mathcal{D}_{\tau \alpha}^{1/2}e^{-\dnt},
\end{align*}
Therefore,
\begin{align*}
T_{1;LH} \lesssim \myr{\nu^{-1/2}} \mathcal{E}^{1/2} \mathcal{D}_{\tau \alpha}^{1/2}\mathcal{D}_{\gamma}^{1/2}e^{-3\dnt},
\end{align*}
which is consistent with \eqref{est:altau} and Theorem \ref{thm:NLNzero}. 
For the $HL$ contribution note that we have instead 
\begin{align*}
T_{1;HL} & \lesssim \sum_{ \substack{k,k' \neq 0\\
  k\nq k'}} \norm{\nu^{1/3} \abs{k}^{m-1/3} \p_y \omega_{k}}_{L^2} \norm{\nu^{1/3} \abs{k-k'}^{m+2/3\mh{-\delta}} \partial_y \phi_{k-k'}}_{L^\infty} \norm{\partial_y \omega_{k'}}_{L^2}. 
\end{align*}
By Lemma \ref{lem:dxphiL2t} we have
\begin{align*}
T_{1;HL} \lesssim \nu^{-1/2}\mathcal{E}^{1/2} \mathcal{D}_{\tau\alpha}^{1/2} \mathcal{D}^{1/2}_\gamma e^{-3\dnt},
\end{align*}
which is consistent with \eqref{est:altau} and Theorem \ref{thm:NLNzero}. This completes the treatment of $T_{\nq\nq; 1}$.

We turn to $T_{\nq\nq; 2}$ next, which we again begin with a frequency decomposition, 
\begin{align*}
T_{\nq\nq; 2} &= \sum_{ \substack{k,k' \neq 0\\
  k\nq k'}} (\mathbf{1}_{\abs{k-k'} < \abs{k'}/2} + \mathbf{1}_{\abs{k-k'} \geq \abs{k'}/2}) \norm{\nu^{1/3} \abs{k}^{m-1/3} \p_y \omega_{k}}_{L^2}  \nu^{1/3} \abs{k}^{m-1/3\mh{-\delta}} \\
& \qquad \times \norm{(k-k')\phi_{k-k'}}_{L^\infty} \norm{\partial_y^2 \omega_{k'}}_{L^2} \\
& =: T_{2;LH} + T_{2;HL}. 
\end{align*}
To treat the first term we use that on the support of the summand $\abs{k-k'} + \abs{k} \lesssim \abs{k'}$ to obtain 
\begin{align*}
T_{2;LH} & \lesssim \sum_{ \substack{k,k' \neq 0\\
  k\nq k'}} \mathbf{1}_{\abs{k-k'} < \abs{k'}/2} \norm{\nu^{1/3} \abs{k}^{m-1/3} \p_y \omega_{k}}_{L^2}   \norm{\abs{k-k'} \phi_{k-k'}}_{L^\infty} \norm{\nu^{1/3}\abs{k'}^{m-1/3}\partial_y^2 \omega_{k'}}_{L^2}. 
\end{align*}
Using Lemma \ref{lem:dxphiL2t} and $m > 1/2$ we have
\begin{align*}
T_{2;LH} \lesssim \nu^{-1/2} \mathcal{E}_{\neq}^{1/2} \mathcal{D}_\tau^{1/2} \mathcal{D}_\alpha^{1/2}e^{-3\dnt}, 
\end{align*}
which is consistent with Theorem \ref{thm:NLNzero}. 
For $T_{2;HL}$ notice that we similarly have (again using Lemma \ref{lem:dxphiL2t} and $m > 1/2$), 
\begin{align*}
T_{2;HL} & \lesssim \sum_{ \substack{k,k' \neq 0\\
  k\nq k'}} \mathbf{1}_{\abs{k-k'} \geq \abs{k'}/2} \norm{\nu^{1/3} \abs{k}^{m-1/3} \p_y \omega_{k}}_{L^2}   \norm{\abs{k-k'}^{m + 1\mh{ - \delta}} \phi_{k-k'}}_{L^\infty} \norm{\partial_y (\abs{k'}^{-1/3} \nu^{1/3}\partial_y) \omega_{k'}}_{L^2} \\
& \lesssim \nu^{-1/2} \mathcal{E}^{1/2} \mathcal{D}^{1/2}_\tau \mathcal{D}^{1/2}_\gamma e^{-3\dnt},
\end{align*}
which is consistent with \eqref{est:altau}, Theorem \ref{thm:NLNzero}. 

\ifx Turn to $T_3$ next, for which we mirror the treatment of the $\gamma$ terms.
As in previous steps we begin with a frequency decomposition
\begin{align*}
T_3 & \lesssim \sum_{ \substack{k,k' \neq 0\\
  k\nq k'}} \left( \mathbf{1}_{\abs{k-k'} < \abs{k'}/2} + \mathbf{1}_{\abs{k-k'} \geq \abs{k'}/2} \right) \norm{\nu^{1/3} \abs{k}^{m-1/3} \p_y \omega_{k}}  \nu^{1/3} k^{m-1/3\mh{-\delta}} \norm{\partial_y^2 \phi_{k-k'}}_{L^\infty} \norm{k' \omega_{k'}}_{L^2} \\
& = T_{3;LH} + T_{3;HL}. 
\end{align*}
To estimate $T_{3;HL}$ we have (using $m > 2/3$, Gagliardo-Nirenberg-Sobolev, and elliptic regularity followed by Cauchy-Schwarz)
\begin{align*}
T_{3;HL} & \lesssim \sum_{ \substack{k,k' \neq 0\\
  k\nq k'}} \mathbf{1}_{\abs{k-k'} \geq \abs{k'}/2} \norm{\nu^{1/3} \abs{k}^{m-1/3} \p_y \omega_{k}}_{L^2}  \nu^{1/3} k^{m-1/2\mh{-\delta}} \norm{\partial_y \omega_{k-k'}}_{L^2}^{1/2} \norm{\abs{k-k'}^{1/3} \omega_{k-k'}}_{L^2}^{1/2} \norm{ \abs{k'}^{m+1/3} \omega_{k'}}_{L^2} \\
& \lesssim \nu^{1/3-1/4-1/12-1/6} \mathcal{E}^{1/2} \mathcal{D}_\gamma^{1/4} \mathcal{D}_{\beta}^{3/4} =  \nu^{-1/6} \mathcal{E}^{1/2} \mathcal{D}_\gamma^{1/4} \mathcal{D}_{\beta}^{3/4},
\end{align*}
which is consistent with Theorem \ref{thm:NLNzero}. 
To treat $T_{3;LH}$, first note that for any $\delta' > 0$, by Gagliardo-Nirenberg-Sobolev, elliptic regularity, Cauchy-Schwarz and $m > 2/3$, 
\begin{align*}
\sum_k \norm{\partial_y^2 \phi_{k}}_{L^\infty} & \lesssim \sum_k \abs{k}^{-1/2-\delta'} \abs{k}^{2/3-\delta'}\norm{k^{-1/3}\partial_y \omega_{k}}_{L^2}^{1/2} \norm{\omega_k}_{L^2}^{1/2} \\
& \lesssim \nu^{-1/6} \mathcal{E}^{1/2}. 
\end{align*}
Therefore, 
\begin{align*}
T_{3;LH} & \lesssim \sum_{ \substack{k,k' \neq 0\\
  k\nq k'}} \mathbf{1}_{\abs{k-k'} < \abs{k'}/2} \nu^{2/3} \norm{\abs{k}^{m} \p_y \omega_{k}}_{L^2} \norm{\partial_y^2 \phi_{k-k'}}_{L^\infty} \norm{\abs{k'}^{m+1/3} \omega_{k'}}_{L^2} \\
& \lesssim \nu^{2/3-1/2-1/6-1/6} \mathcal{E}^{1/2} \mathcal{D}_{\gamma}^{1/2} \mathcal{D}_\beta^{1/2} = \nu^{-1/6} \mathcal{E}^{1/2} \mathcal{D}_{\gamma}^{1/2} \mathcal{D}_\beta^{1/2}, 
\end{align*}
which is consistent with Theorem \ref{thm:NLNzero}. This completes the treatment of $T_3$. 
\fi
Turning to $T_{\nq\nq;\tau 3}$ we begin with a frequency decomposition
\begin{align*}
T_{\nq\nq;\tau 3} & \lesssim \sum_{ \substack{k,k' \neq 0\\
  k\nq k'}} \left( \mathbf{1}_{\abs{k-k'} < \frac{\abs{k'}}{2}} + \mathbf{1}_{\abs{k-k'} \geq \frac{\abs{k'}}{2}} \right) \norm{\nu^{1/3} |k|^{m-1/3} \pa_y\mathfrak{J}_k\p_y \omega_{k}}_{L^2} \nu^{1/3} \abs{k}^{m-1/3\mh{-\delta}} \norm{\partial_y \phi_{k-k'}}_{L^\infty} \norm{k' \omega_{k'}}_{L^2} \\
& = T_{3;LH} + T_{3;HL}.
\end{align*}
To treat $T_{3;LH}$ first note that by Gagliardo-Nirenberg-Sobolev, $m >1/2$, and elliptic regularity we have 
\begin{align*}
\sum_{k\nq0} \norm{\na_k \phi_{k}}_{L^\infty} \lesssim \sum_{k\nq 0} \norm{\omega_k}_{L^2}^{1/2} \norm{\na_k \phi_k}_{L^2}^{1/2} \lesssim \mathcal{E}_k^{1/2}e^{-\dnt}. 
\end{align*}
Combining this, the multiplier Lemma \ref{lem:BoundI}, \ref{lem:CommIpy}, \ref{lem:H}, Lemma \ref{lem:komegaD}, Gagliardo-Nirenberg inequality and Young's convolution inequality yields
\begin{align*}
T_{3;LH} & \lesssim \sum_{ \substack{k,k' \neq 0\\
  k\nq k'}} \mathbf{1}_{\abs{k-k'} < \abs{k'}/2} \nu^{1/3}\norm{\nu^{1/3} |k|^{m-1/3} \na_k\p_y \omega_{k}}_{L^2}  \norm{\partial_y \phi_{k-k'}}_{L^\infty} \norm{ \abs{k'}^{m+2/3} \omega_{k'}}_{L^2} \\
& \lesssim \nu^{- 1/2} \mathcal{E}^{1/2} \mathcal{D}_{\alpha}^{1/2}\mathcal{D}_\gamma^{1/4}\mathcal{D}_\beta^{1/4}e^{-3\dnt}, 
\end{align*} 
which is consistent with Theorem \ref{thm:NLNzero}. 
Finally, to treat $T_{4;HL}$ we have
\begin{align*}
T_{3;HL} & \lesssim \sum_{ \substack{k,k' \neq 0\\
  k\nq k'}} \mathbf{1}_{\abs{k-k'} \geq \abs{k'}/2} \norm{\nu^{1/3} k^{m-1/3} \p_y^2 \omega_{k}}_{L^2} \nu^{1/3} \abs{k-k'}^{m-1/3-\delta} \norm{\partial_y \phi_{k-k'}}_{L^\infty} \norm{ \abs{k'}^{m+1/3} \omega_{k'}}_{L^2} \\
& \lesssim \nu^{1/3-1/6-1/2} \mathcal{D}_{\alpha}^{1/2} \mathcal{D}_\beta^{1/2} \mathcal{E}^{1/2}e^{-3\dnt},
\end{align*}
where we used Gagliardo-Nirenberg-Sobolev and elliptic regularity to deduce 
\begin{align*}
\sum_k \abs{k}^{m-1/3\mh{-\delta}}\norm{\partial_y \phi_{k}}_{L^\infty} \lesssim \sum_k \abs{k}^{-1/2\mh{-\delta}} \abs{k}^{m}\norm{\omega_k}_{L^2}^{1/2} \norm{\abs{k}^{1/3}\partial_y \phi_k}_{L^2}^{1/2} \lesssim \mathcal{E}^{1/2}_\nq e^{-\dnt}. 
\end{align*}
This is consistent with \eqref{est:altau}, Theorem \ref{thm:NLNzero}, which then completes the treatment of $T_{\nq\nq;3}$. Hence the proof of Lemma \ref{lem:altau} is completed.

\subsubsection{The $\beta$ Contributions}
In this subsection, we prove Lemma \ref{lem:beta}. We 
decompose the $\mathcal{NL}_\beta$-term in \eqref{NL} as  follows
\begin{align*}e^{-2\dnt}\mathcal{NL}_\beta & =- c_\beta \sum_{k\nq 0} \nu^{1/3} \abs{k}^{2m-4/3} \left( \Re \brak{ik \omega_k,  \p_y (\grad^\perp \phi \cdot \grad \omega)_k} + \Re \brak{ik (\grad^\perp \phi \cdot \grad \omega)_k , \p_y \omega_k} \right)\\
& = T_{\beta;1} + T_{\beta;2}. 
\end{align*}
Each term can be treated similarly due to integration by parts, hence we focus only on $T_{\beta;1}$ without loss of generality. \siming{We further note that we will only bound the absolute values of  these terms, hence the sign is not important. }
First, we separate the contributions of the zero and non-zero frequencies
\begin{align}\n
T_{\beta;1} & = -c_\beta \sum_{k \neq 0} \sum_{k'=-\infty}^\infty \left(\mathbf{1}_{k=k' , k' \neq 0} + \mathbf{1}_{k\neq k', k' = 0} + \mathbf{1}_{k \neq k', k' \neq 0} \right) \nu^{\frac{1}{3}} \abs{k}^{{2m}-\frac{4}{3}}  \Re \brak{ik \omega_k, \p_y \lf(\grad_{k-k'}^\perp \phi_{k-k'} \cdot \grad_{k'} \omega_{k'}\rg)} \\
& =: T_{0\neq} + T_{\neq 0} + T_{\neq \neq}. \label{Tbeta}
\end{align}

\noindent
\textbf{Estimate of the $T_{0\neq}$ term in \eqref{Tbeta}:} \\
Distributing the derivative and integrating by parts we have
\begin{align*}
|T_{0\neq}| 
& = \bigg|\sum_{k \neq 0}  \nu^{1/3} \abs{k}^{-4/3} \abs{k}^{2m} \Re \brak{ik \omega_k, \p_y \phi_{0} \p_y ik \omega_{k})}\siming{-}  \sum_{k \neq 0}  \nu^{1/3} \abs{k}^{-4/3} \abs{k}^{2m} \Re \brak{ik \omega_k, \p_y^2 \phi_{0} ik \omega_{k})}\bigg| \\
& = \bigg| \frac{1}{2}\sum_{k \neq 0}  \nu^{1/3}  \abs{k}^{2m+ 2/3} \Re \brak{\omega_k, \p_y^2 \phi_{0}\omega_{k})}\bigg|. 
\end{align*}
By H\"older's inequality we have 
\begin{align*}
|T_{0\neq}| \lesssim \sum_{k \neq 0}  \nu^{1/3} \norm{|k|^{1/3+m} \omega_k}_{L^2}^2 \norm{\p_y^2 \phi_{0}}_{L^\infty}.
\end{align*}
Upon applying Gagliardo-Nirenberg-Sobolev and the $\alpha$ term in the $\mathcal{E}_0$-energy, we have the estimate
\begin{align*}
\norm{\p_y^2 \phi_{0}}_{L^\infty} \lesssim \norm{\p_y \omega_{0}}_{L^2}^{1/2}\norm{\omega_{0}}_{L^2}^{1/2} \lesssim \nu^{-1/6} \mathcal{E}_0^{1/2}e^{ -\delta_\ast\nu t}.
\end{align*}
Combining the bounds above yields that,  
\begin{align*}
|T_{0\neq}| \lesssim \nu^{-1/6} \mathcal{E}_0^{1/2} \mathcal{D}_\beta e^{-2\dnt -\delta_\ast\nu t}, 
\end{align*}
which is consistent with \eqref{est:beta}, Theorem \ref{thm:NLNzero}. 

\noindent
\textbf{Estimate of the $T_{\neq 0}$ term in \eqref{Tbeta}:} \\
Integrating by parts and using H\"older's inequality yield that
\begin{align*}
 |T_{\neq 0}| & = \bigg|\sum_{k \neq 0}  \nu^{1/3}  \abs{k}^{2m-4/3}\Re \brak{ik \p_y \omega_k, (ik \phi_{k} \p_y \omega_{0})} \bigg|\\
& \lesssim  \sum_{k \neq 0}   \abs{k}^{-1/3}  \norm{\abs{k}^{ m}\p_y \omega_k}_{L^2} \norm{\abs{k}^{ m+1} \phi_{k}}_{L^\infty} (\nu^{1/3}\norm{\p_y \omega_{0}}_{L^2})\lesssim \nu^{-1/2}\mathcal{D}_{\gamma}^{1/2}\mathcal{D}_\tau^{1/2}\mathcal{E}_0^{1/2}  e^{-2\dnt-\delta_\ast\nu t}, 
\end{align*}
where in the last line we used the $\alpha$ term in the $\mathcal{E}_0$-energy \eqref{mcl_E_0} and Lemma \ref{lem:dxphiL2t}. 

\noindent
\textbf{Estimate of the $T_{\neq\neq}$ term in \eqref{Tbeta}:} \\
This is the hardest term, however we treat it in someways similar to the first two terms.
First we separate out the different components of the velocity fields 
\begin{align}\n
|T_{\neq\neq}| & \leq\bigg| \nu^{1/3} \sum_{ \substack{k,k' \neq 0\\
  k\nq k'}} \Re \brak{\abs{k}^{2m - 4/3} ik \omega_{k}, \partial_y (i(k-k') \phi_{k-k'} \partial_y\omega_{k'})} \bigg|\\
& \quad + \bigg|\nu^{1/3} \sum_{\siming{ \substack{k,k' \neq 0\\
  k\nq k'}}} \Re \brak{\abs{k}^{2m - 4/3} ik \omega_{k}, \partial_y (\partial_y \phi_{k-k'} ik'\omega_{k'})}\bigg| =: T_x +T_y.\label{Tb_nqnq}
\end{align} 
The $T_x$ term is consistent using the $\alpha$ term in the energy along with the dissipation, as long as $m > 1/2$.
Indeed, we have for $m  > 1/2$ using Lemma \ref{lem:dxphiL2t}, 
\begin{align*}
T_x & \lesssim \sum_{ \substack{k,k' \neq 0\\
  k\nq k'}} \norm{\nu^{1/3} |k|^{-1/3} \partial_y |k|^m \omega_k}_{L^2}(\abs{k-k'}^m + \abs{k'}^m) \norm{\abs{k-k'}\phi_{k-k'}}_{L^\infty} \norm{\partial_y \omega_{k'}}_{L^2} \\
& \lesssim \frac{\mathcal{E}^{1/2}}{\sqrt{\nu}}\mathcal{D}_\tau^{1/2} \mathcal{D}_\gamma^{1/2}e^{-3\dnt}, 
\end{align*}
which is consistent with \eqref{est:beta} and Theorem \ref{thm:NLNzero}. 
Next, consider the more formidable $T_y$ term.
We first distribute the derivative to treat each contribution separately
\begin{align}\n
T_y  & \leq\bigg| \nu^{1/3} \sum_{ \substack{k,k' \neq 0\\
  k\nq k'}} \Re \brak{\abs{k}^{2m - 4/3} ik \omega_{k}, \partial_y^2 \phi_{k-k'} ik'\omega_{k'}} \bigg|\\
& \quad + \bigg|\nu^{1/3} \sum_{ \substack{k,k' \neq 0\\
  k\nq k'}}\Re \brak{\abs{k}^{2m - 4/3} ik \omega_{k}, \partial_y \phi_{k-k'} ik'\partial_y\omega_{k'}}\bigg| =: T_{y,1} + T_{y,2}.\label{Tbeta_y} 
\end{align}
Consider the first term.
We split it into $HL$ and $LH$ contributions based on the $x$-frequency
\begin{align}
\n T_{y,1} & = \bigg|\nu^{1/3} \sum_{ \substack{k,k' \neq 0\\
  k\nq k'}} \left(\mathbf{1}_{\abs{k-k'} > \abs{k'}/2} + \mathbf{1}_{\abs{k'-k} < \abs{k'}/2} \right)  \Re \brak{\abs{k}^{2m - 4/3} ik  \omega_k, \partial_y^2 \phi_{k-k'} ik'\omega_{k'}}\bigg| \\
& = :|T_{HL} + T_{LH}|.\label{Tbeta_y1}  
\end{align}
Consider first the $T_{LH}$ term in \eqref{Tbeta_y1}. Using that $\abs{k} \leq\frac{3 }{2}\abs{k'}$ on the support of the summand and $m > 2/3$, we have
\begin{align*}
|T_{LH}| & \lesssim \sum_{ \substack{k,k' \neq 0\\
  k\nq k'}}  \mathbf{1}_{\abs{k-k'} \leq \frac{1}{2}\abs{k'}} \nu^{1/3} \norm{\abs{k}^{m - 2/3} k \omega_k}_{L^2} \norm{\abs{k'}^{m - 2/3} k' \omega_{k'}}_{L^2} \norm{\partial_y^2 \phi_{k-k'}}_{L^\infty}.
\end{align*}
Then we observe that by \eqref{ineq:d2phiAlpha} and Young's convolution inequality, 
\begin{align*}
|T_{LH}| & \lesssim \bigg(\sum_{k \neq 0}  \nu^{1/3} \norm{\abs{k}^{m +1/3}   \omega_k}_{L^2}^2\bigg)^{1/2}\bigg(\nu^{1/3}\sum_{k'\nq 0} \norm{\abs{k'}^{m + 1/3} \omega_{k'}}_{L^2}^2\bigg)^{1/2} \bigg(\sum_{k\nq0}\norm{\partial_y^2 \phi_{k}}_{L^\infty}\bigg)\\ \lesssim& {\nu^{-1/ 6}} \mathcal{E}^{1/2} \mathcal{D}_\beta e^{-3\dnt},
\end{align*}
which is consistent with \eqref{est:beta} and Theorem \ref{thm:NLNzero}. 
Next consider the $T_{HL}$ term in \eqref{Tbeta_y1}, which we instead estimate as follows, using that $m > 2/3$ and $\abs{k'} + \abs{k} \lesssim \abs{k-k'}$ on the support of the summand, 
\begin{align*} 
|T_{HL} |& \lesssim \nu^{1/3} \sum_{ \substack{k,k' \neq 0\\
  k\nq k'}}\mathbf{1}_{\frac{1}{2}\abs{k'} \leq \abs{k-k'}} \norm{\abs{k}^{m - 2/3} k \omega_k}_{L^2} \norm{k' \omega_{k'}}_{L^\infty} \norm{\abs{k-k'}^{m - 2/3} \partial_y^2 \phi_{k-k'}}_{L^2} \\
  & \lesssim \nu^{1/3} \sum_{ \substack{k,k' \neq 0\\
  k\nq k'}}\mathbf{1}_{\frac{1}{2}\abs{k'} \leq\abs{k-k'}} \norm{\abs{k}^{m + 1/3}  \omega_k}_{L^2} \norm{\omega_{k'}}_{L^\infty}  \norm{\abs{k-k'}^{m + 1/3} \partial_y^2 \phi_{k-k'}}_{L^2}. 
\end{align*}
From here we use essentially the same argument as we did for the $LH$ term.
First we note that by elliptic regularity, 
\begin{align*}
  \norm{\abs{k-k'}^{m + 1/3} \partial_y^2 \phi_{k-k'}}_{L^2} & \lesssim \norm{\abs{k-k'}^{m + 1/3} \omega_{k-k'}}_{L^2},
\end{align*}
and that by Gagliardo-Nirenberg-Sobolev (and $m >2/3$
), we have 
\begin{align*}
 &\sum_{k\neq 0} \norm{\omega_{k}}_{L^\infty} \lesssim \sum_{k\neq0} \norm{\omega_{k}}_{L^2}^{1/2} \norm{\partial_y \omega_{k}}_{L^2}^{1/2}  \lesssim \nu^{-1/6}\sum_{k\neq0} \abs{k}^{-m+1/6} \norm{|k|^m\omega_{k}}_{L^2}^{1/2} \norm{\nu^{1/3}\abs{k}^{m-1/3}\partial_y \omega_{k}}_{L^2}^{1/2}\\
&\lesssim \nu^{-1/6}\bigg(\sum_{k\nq 0 }|k|^{-2m+1/3}\bigg)^{1/2}\bigg(\sum_{k\neq0}\norm{|k|^m\omega_{k}}_{L^2}^2\bigg)^{1/4} \bigg(\sum_{k\nq 0}\norm{\nu^{1/3}\abs{k}^{m-1/3}\partial_y \omega_{k}}_{L^2}^2\bigg)^{1/4} \lesssim \nu^{-1/6} \mathcal{E}^{1/2}e^{-\dnt}.  
\end{align*}
Therefore, we similarly have
\begin{align*}
|T_{HL} |\lesssim \nu^{-1/6} \mathcal{E}^{1/2} \mathcal{D}_\beta e^{-3\dnt},
\end{align*}
which is consistent with \eqref{est:beta} and Theorem \ref{thm:NLNzero}. This completes the treatment of the $T_{y,1}$ term. 

Next, we consider the $T_{y,2}$ term in \eqref{Tbeta_y}, for which we introduce a commutator (writing completely on the physical-side for a moment): 
\begin{align}\n
|T_{y,2}| & = \bigg|\nu^{1/3} \brak{\abs{\partial_x}^{m - 2/3} \partial_x \omega_{\neq}, \partial_y \phi_{\neq} \partial_y \abs{\partial_x}^{m - 2/3} \partial_x \omega_{\neq}} \\
& \qquad + \nu^{1/3} \brak{\abs{\partial_x}^{m - 2/3} \partial_x \omega_{\neq}, [\partial_y \phi_{\neq},\abs{\partial_x}^{m - 2/3}] \partial_y \partial_x \omega_{\neq}} \bigg|
  =: |T_{m} + T_c|.\label{Tmc} 
\end{align}
Here, ``$m$'' is ``main'' and ``$c$'' is ``commutator''. 
For the first term, we integrate by parts and  use H\"older's inequality after expanding in Fourier to obtain, 
\begin{align*}
|T_m |& = \bigg|\frac{1}{2}\nu^{1/3} \brak{\abs{\partial_x}^{m - 2/3} \partial_x \omega_{\neq}, \partial_y^2 \phi_{\neq} \abs{\partial_x}^{m - 2/3} \partial_x \omega_{\neq}}\bigg| \\ 
&\lesssim \nu^{1/3} \sum_{\substack{k,k'\nq0;\\ k\nq k'}}\norm{ \abs{k}^{m+1/3} \omega_k}_{L^2}\norm{ \abs{k'}^{m+1/3} \omega_{k'}}_{L^2} \norm{\partial_y^2 \phi_{k-k'}}_{L^\infty}. 
\end{align*} 
Using \eqref{ineq:d2phiAlpha} again, we therefore have 
\begin{align*}
|T_m| \lesssim \nu^{-1/6} \mathcal{E}^{1/2} \mathcal{D}_\beta e^{-3\dnt}, 
\end{align*}
which is consistent with \eqref{est:beta} and Theorem \ref{thm:NLNzero}. 

Finally we turn to the term with the commutator in \eqref{Tmc}, $T_c$, for which we expand on the Fourier-side and use another frequency decomposition
\begin{align*}
|T_c| & \lesssim  \nu^{1/3} \sum_{ \substack{k,k' \neq 0\\
  k\nq k'}} (\mathbf{1}_{\abs{k-k'} < \abs{k'}/2} + \mathbf{1}_{\abs{k-k'} \geq \abs{k'}/2}) \int_{-1}^1 \abs{k}^{m - 2/3} \abs{k\omega_k} \abs{\abs{k}^{m - 2/3} - \abs{k'}^{m-2/3}} \abs{\partial_y \phi_{k-k'}  k'\partial_y\omega_{k'}} \dee y \\
  & =: T_{c;LH} + T_{c;HL}. 
\end{align*}
For the $T_{c;LH}$ term we use the commutator estimate (which uses $m > 2/3$) 
\begin{align*}
\abs{\abs{k}^{m - 2/3} - \abs{k'}^{m-2/3}} \lesssim \abs{k'}^{m-2/3-1} \abs{k-k'},\quad k'\nq 0, 
\end{align*}
which gives (using also $\abs{k-k'} \lesssim \abs{k'}$ on the support of the summand)
\begin{align*}
  T_{c;LH} & \lesssim \nu^{-1/6}\sum_{ \substack{k,k' \neq 0\\
  k\nq k'}}\mathbf{1}_{\abs{k-k'} < \abs{k'}/2} \norm{\nu^{1/6}\abs{k}^{m +1/3} \omega_k}_{L^2} \norm{\abs{k-k'}^{2/3}\partial_y \phi_{k-k'}}_{L^\infty}\norm{ \nu^{1/3} \abs{k'}^{m-1/3}\partial_y\omega_{k'}}_{L^2}. 
\end{align*}
Then, we observe the following estimate, which follows from the Gagliardo-Nirenberg-Sobolev inequality, \siming{
\begin{align*}
\sum_{k\nq0} \norm{\abs{k}^{2/3}\partial_y \phi_{k}}_{L^\infty} & \lesssim \sum_{k\nq0} \norm{\abs{k}^{1/3}\partial_y^2 \phi_{k}}_{L^2}^{1/2} \norm{|k|\partial_y \phi_k}_{L^2}^{1/2} +\sum_{k\nq0} \norm{\abs{k}^{2/3}\partial_y  \phi_{k}}_{L^2} \\
  &\lesssim \sum_{k\nq0} \abs{k}^{-m-1/6}\norm{\abs{k}^{m+1/3} \omega_k}_{L^2} 
  \lesssim \nu^{-1/6}\mathcal{D}^{1/2}_\beta e^{-\dnt}. 
\end{align*}}
Therefore, we obtain
\begin{align*}
T_{c;LH} \lesssim \nu^{\myr{-1/3}} \mathcal{E}^{1/2}_\al \mathcal{D}_\beta e^{-3\dnt},%
\end{align*}
which is consistent with \eqref{est:beta}, Theorem \ref{thm:NLNzero}. 

Finally, we consider $T_{c,HL}$, for which the commutator is not relevant. 
Using that $0 < m-1/3 < 1$ and $\abs{k'} + \abs{k} \lesssim \abs{k-k'}$ on the support of the summand, we have 
\begin{align*}
T_{c,HL} & \lesssim \nu^{1/3} \sum_{ \substack{k,k' \neq 0\\
  k\nq k'}} \mathbf{1}_{\abs{k-k'} \geq \abs{k'}/2} \abs{k}^{m - 2/3} \abs{k\omega_k} \abs{\abs{k}^{m - 2/3} - \abs{k'}^{m-2/3}} \abs{\partial_y \phi_{k-k'}  k'\partial_y\omega_{k'}} \\
& \lesssim \sum_{ \substack{k,k' \neq 0\\
  k\nq k'}} \nu^{1/3} \norm{|k|^{m + 1/3} \omega_k}_{L^2} \norm{\abs{k-k'}^{m+1/3} \partial_y \phi_{k-k'}}_{L^\infty} \norm{\partial_y \omega_{k'}}_{L^2}.
\end{align*}
Next, note $m > 1/2$ and Gagliardo-Nirenberg-Sobolev imply the following estimates 
\begin{align*}
\sum_{k\nq 0} \norm{\partial_y \omega_{k}}_{L^2} & \lesssim \bigg(\sum_{k\nq0} \norm{\abs{k}^m \partial_y \omega_k}_{L^2}^2\bigg)^{1/2}\approx \nu^{-1/2}\mathcal D_{\gamma}^{1/2}e^{-\dnt}, \\
\norm{\abs{k}^{m+1/3} \partial_y \phi_{k}}_{L^\infty} &  \lesssim   \norm{\abs{k}^{m}\partial_y^2 \phi_{k}}_{L^2}^{1/2} \norm{|k|^{m+2/3}\partial_y \phi_k}_{L^2}^{1/2} +  \norm{\abs{k}^{m+1/3}\partial_y  \phi_{k}}_{L^2} \lesssim  \norm{\abs{k}^m \omega_k}_{L^2}, 
\end{align*} 
and therefore it follows that  
\begin{align*}
T_{c,HL} & \lesssim \nu^{-1/3}\mathcal{E}^{1/2}\mathcal{D}_\beta^{1/2} \mathcal{D}_\gamma^{1/2}e^{-3\dnt}, 
\end{align*}
which is consistent with \eqref{est:beta} and  Theorem \ref{thm:NLNzero}.
This completes the treatment of the $\mathcal{NL}_\beta$ term. 
\siming{
Finally, we provide the proof of the main proposition  \ref{prop:MainBoot}, which directly implies Theorem \ref{thm:main}:
\begin{proof}[Proof of Proposition  \ref{prop:MainBoot}]
We combine the decomposition \eqref{ddtE0}, \eqref{ddtEnq}, Lemma \ref{lem:NLzero} and Theorem \ref{thm:NLNzero} to obtain that 
\begin{align*}
\frac{d}{dt}\mathcal{E}\leq -\frac{1}{2}\mathcal{D}_0-6\delta_\ast \nu \mathcal{E}_0 -4\delta_\ast  \mathcal{D}_\nq-6\delta_\ast \nu^{1/3}\sum_{k \neq 0}  e^{2\delta_\ast \nu^{1/3} t} \abs{k}^{2m + 2/3} E_k[\omega_k]+\frac{C\mathcal{E}^{1/2}}{\sqrt{\nu}}\mathcal{D}.
\end{align*}
Recalling the definition of $\mathcal{D}=\mathcal{D}_0+\mathcal{D}_{\nq}+\mathcal{D}_E$ \eqref{mcl_D}, we observe that for $\delta_\ast \leq \frac{1}{16}$, the estimate above implies \eqref{ineq:MainEst}. Integration in time yields the main result. 
\end{proof}
}